\documentclass[12pt]{amsart}
\usepackage{pifont}
\usepackage{amsfonts}
\usepackage{amsmath}
\usepackage{amssymb}
\usepackage{amsmath,bm}
\usepackage{cite}
\usepackage{amsmath,enumerate}
\usepackage{stmaryrd}
\usepackage{mathrsfs}
\usepackage{color}
\usepackage{mathtools,enumerate}

\usepackage[pdftex]{graphicx}
\usepackage{epstopdf}
\usepackage{caption}
\usepackage{color}
\newtheorem{theorem}{Theorem}[section]

\newtheorem{lemma}[theorem]{Lemma}

\newtheorem{example}{Example}[section]
\theoremstyle{definition}

\newcommand{\Real}{\mathbb{R}}
\begin{document}

\title[Transition semi-wave of reaction diffusion equations with free boundaries]{Transition semi-wave solutions of reaction diffusion equations with free boundaries}

\author{Xing Liang}
\author {Tao Zhou}
\address{School of Mathematical Sciences and Wu Wen-Tsun Key Laboratory of Mathematics, University of Science and Technology of China, Hefei, Anhui 230026, China}
\date{5th Sep., 2018 }
\maketitle

\begin{abstract}
In this paper, we define the transition semi-wave solution  (c.f. Definition \ref{def2.1}) of the following reaction diffusion equation with free boundaries
\begin{equation}\label{0.1}
\left\{
   \begin{aligned}
  u_{t}=u_{xx}+f(t,x,u),\ \ &t\in\Real, x<h(t),\\
   u(t,h(t))=0,\ \ &t\in\Real,\\
   h^{\prime}(t)=-\mu u_{x}(t,h(t)),\ \ &t\in\Real,\\
   \end{aligned}
   \right.
\end{equation}
In the homogeneous case, i.e., $f(t,x,u)=f(u)$, under the hypothesis $$
f(u)\in {C}^{1}([0,1]), f(0)=f(1)=0, f^{\prime}(1)<0, f(u)<0\ \text{for}\ u>1,
$$ we prove that the 
semi-wave connecting $1$ and $0$ of \eqref{0.1} is unique provided it exists. Furthermore, we prove that any bounded transition semi-wave connecting $1$ and 0 is exactly the semi-wave. 

In the cases where $f$ is KPP-Fisher type and almost periodic in time (space), i.e., 
$f(t,x,u)=u(c(t)-u)$ (resp. $u(a(x)-u)$) with $c(t)$ (resp. $a(x)$) being almost periodic, 
applying totally different method, we also prove any bounded transition semi-wave connecting the unique almost periodic positive solution of $u_{t}=u(c(t)-u)$ (resp. $u_{xx}+u(a(x)-u)=0$) and $0$
is exactly the unique almost periodic semi-wave of \eqref{0.1}.
Finally, we provide an example of the heterogeneous equation to show the existence of the transition semi-wave without any global mean speeds.
\end{abstract}


\section{Introduction}

Since the pioneer works of Fisher \cite{F1} and Kolmogorov, Petrovsky, Piskunov \cite{K1} in 1937,
there have been many studies on the spreading phenomena,
especially the traveling wave of reaction diffusion equations.
In the last decades, spreading phenomena in heterogeneous media got more attention from mathematicians.
For the equations in spatially periodic media, \cite{SKT} and \cite{X} gave the definition of the spatially periodic traveling waves independently,
and then \cite{HZ} proved the existence of the spatially periodic traveling waves of Fisher-KPP equations in the distributional sense. From then on, there were many works about traveling waves in spatially periodic media appear, see e.g. \cite{ BH,BHR,BHN, DHZ1,FZ, LiZh, Wein2002 }.  
Matano \cite{M} defined the traveling wave in spatially recurrent diffusive media and also 
discussed its existence, uniqueness and stability in the bistable case. Applying the same idea as in \cite{M},
Shen defined the traveling waves in time almost periodic media and
studied the existence, uniqueness and stability of the traveling waves in bistable case in \cite{S,S2}.
It is known that traveling wave solutions are special examples of the entire solutions that are defined in the whole space and for
all time $t\in\mathbb{R}$. There were many works about the entire solutions, see \cite{CGH,CGN, HN1,HN2, SLW,WLR} and references therein.

In \cite{BH07,BH12}, Berestycki and Hamel introduced a generalized concept of the traveling wave,
named the transition wave, which is still a special kind of the entire solutions and describes a general class of wave-like solutions for reaction diffusion
equations in general heterogeneous media.
Then there were many further works on transition waves for random dispersal,
see \cite{BMH,BGW,BW,DHZ,NR,S3,H,HR,Z} and references therein.
Existence of transition waves of the heterogeneous Fisher-KPP equations with nonlocal dispersal under some assumptions can be found in \cite{LZ} and \cite{SS}.
In \cite{CS}, transition waves of the discrete Fisher-KPP equation in time heterogeneous media were studied.

After the work of Du and Lin \cite{DLi}, spreading phenomena of the following reaction diffusion equation with free boundaries were studied:
\begin{equation}\label{1.1}
\left\{
   \begin{aligned}
  u_{t}=u_{xx}+f(t,x,u),\ \ &t\in\Real, x<h(t),\\
   u(t,h(t))=0,\ \ &t\in\Real,\\
   h^{\prime}(t)=-\mu u_{x}(t,h(t)),\ \ &t\in\Real,\\
   \end{aligned}
   \right.
\end{equation}
where $x=h(t)$ is the moving boundary, and $\mu\in(0,+\infty)$ is a constant.
The so-called semi-wave, which corresponds to the traveling wave of the Cauchy problem, is used to describe the spreading phenomena.
In the case where $f(t,x,u)=u(a_{1}(t)-b_{1}(t)u)$ for some positive $L$-periodic functions
$a_{1}(t)\ \text{and}\ b_{1}(t)$, the existence and uniqueness of a positive time periodic semi-wave were proved in \cite{DGP}.
In the case where $f(t,x,u)=u(a_{2}(x)-b_{2}(x)u)$ for some positive $L$-periodic functions
$a_{2}(x)\ \text{and}\ b_{2}(x)$, the existence and uniqueness of the positive spatial periodic semi-wave were proved in \cite{DLiang}.
Moreover, \cite{Zh} proved the existence of a positive spatial periodic semi-wave by a different method.
Besides above works, \cite{DLo} gave a complete classification of the spatial-temporal dynamics of the solutions in the case where $f(t,x,u)\equiv f(u)$
is monostable, bistable or combustion.

Recently, the existence and uniqueness of the time almost periodic semi-wave (c.f. Definition \ref{timeTW}) of \eqref{1.1} with the
time almost periodic KPP-Fisher type nonlinearity $f$ has been established in \cite{LLS}.
More precisely, they studied the following problem:
\begin{equation}\label{time}
\left\{
   \begin{aligned}
  u_{t}=u_{xx}+u(c(t)-u),\ \ &t\in\Real, x<h(t),\\
   u(t,h(t))=0,\ \ &t\in\Real,\\
   h^{\prime}(t)=-\mu u_{x}(t,h(t)),\ \ &t\in\Real,\\
   \end{aligned}
   \right.
\end{equation}
where $c(t)$ is an almost periodic function in $t\in\Real$ with
$\lim\limits_{t\to+\infty}\frac{1}{t}\int_{0}^{t}c(s)ds>0$.

 \cite{L} showed the existence and uniqueness of the space almost periodic semi-wave (c.f. Definition \ref{spaceTW}) of \eqref{1.1} with the space almost periodic KPP-Fisher type nonlinearity $f$.
More precisely, \cite{L} studied
\begin{equation}\label{space}
\left\{
   \begin{aligned}
  u_{t}=u_{xx}+u(a(x)-u),\ \ &t\in\Real, x<h(t),\\
   u(t,h(t))=0,\ \ &t\in\Real,\\
   h^{\prime}(t)=-\mu u_{x}(t,h(t)),\ \ &t\in\Real,\\
   \end{aligned}
   \right.
\end{equation}
where $a(x)$ is a positive almost periodic function in $x\in\Real$.

Until now, there is no work considering transition wave solutions of \eqref{1.1}.

In the present paper, we consider the transition semi-waves for problem \eqref{1.1}.
For any continuous function $h$ on $ \mathbb{R}$, we denote $\Omega_{h}=\{(t,x): t\in \Real, x<h(t)\}.$
From now on, we will say that an entire solution $(u(t,x),h(t))$ of \eqref{1.1} is positive or bounded provided $u$ is positive or bounded in $\Omega_{h}.$
\newtheorem{defn}{Definition}[section]
\begin{defn}\label{def2.1}
Let $p=p(t,x), (t,x)\in \Real^{2},$ be a positive entire solution of
$u_{t}=u_{xx}+f(t,x,u)$ 
and $(u,h)=(u(t,x),h(t))$ be a positive entire solution of \eqref{1.1}. $(u,h)$ is said to be a transition semi-wave solution (shortly, transition semi-wave)
connecting $p$ and $0$ provided
\begin{equation}\label{defofTranW}
\lim\limits_{x\to-\infty}|u(t,x+h(t))-p(t,x+h(t))|=0
\end{equation}
uniformly in $t\in\Real$.
\end{defn}

The main result of this paper is as follows.
\\
{\bf Main result}

{\it The bounded transition semi-wave $(u(t,x),h(t))$ connecting $p$ and $0$ is
exactly a semi-wave up to a translation in the following three cases:
\begin{enumerate}
\item[(1)] $f(t,x,u)\equiv f(u)$ with $$
f(u)\in {C}^{1}([0,1]), f(0)=f(1)=0, f^{\prime}(1)<0, f(u)<0\ \text{for}\ u>1,
$$ $p(t,x)\equiv1$ and
\eqref{1.1} possesses a semi-wave connecting $1$ and $0$; specially, in this case the semi-wave connecting $1$ and $0$ is unique;
\item[(2)]  $f(t,x,u)=u(c(t)-u)$, where $c(t)$ is an almost periodic function with $\lim\limits_{t\to+\infty}\frac{1}{t}\int_{0}^{t}c(s)ds>0$, and $p(t,x)=u_{c}(t)$ is the unique almost periodic positive solution of $u_{t}=u(c(t)-u)$;
\item[(3)]  $f(t,x,u)=u(a(x)-u)$, where $a(x)$ is an almost periodic positive function, and $p(t,x)=v_{a}(x)$ is the unique 
almost periodic solution solution of $u_{xx}+u(a(x)-u)=0$.
\end{enumerate}
} 

This result means that the free boundary problem may have simpler spatial-temporal dynamics than the respective Cauchy problem.

The paper is organized as follows. In Section 2, we introduce the definitions and state the main results of this paper.
in Section 3, we show some properties of transition semi-wave.
In Section 4, we prove the main result in Case 1.
In Section 5, we prove the main result in Cases 2 and 3.
Finally, in Section 6, we construct  an example of the heterogeneous equation to show the existence of the transition semi-wave without any global mean speeds.

\section{Preliminary: Definitions, notations and results}

In this paper, we always assume that $f:\mathbb{R}\times\mathbb{R}\times[0,+\infty)\to\mathbb{R}$
$$(t,x,u)\to f(t,x,u)$$
is continuous,  of class $C^{\alpha/2,\alpha}(\Real^{2})$ in $(t,x)$ locally uniformly for $u\in\Real$, with $\alpha\in(0,1)$, i.e., 
$\sup\limits_{u\leq M}\|f(\cdot,u)\|_{C^{\alpha/2,\alpha}(\Real^{2})}<+\infty$ for any $M>0$,
that $\partial_{u}f(t,x,u)$ is continuous with $\sup\limits_{(t,x)\in\mathbb{R}^{2},|u|\leq M} \big|\partial_{u}f(t,x,u)\big|<+\infty$ for any $M>0$,
and that  $f(t,x,0)\equiv0$ for any $(t,x)\in\Real^{2}$.

An important notion which is attached to a transition semi-wave is its global mean speed of propagation.

\begin{defn}\label{defgmeans}
We say that the transition semi-wave $(u,h)$ of \eqref{1.1} has a global mean speed $c$ if
$$\frac{|h(t)-h(s)|}{|t-s|}\to c\ \text{as}\ |t-s|\to+\infty.$$
\end{defn}
We will prove in Theorem \ref{thm3.1} that the global mean speed is unique among a certain class of transition semi-waves,
and any such waves can be compared up to shift. However, the global mean speed  does not always exist in general.
In the last section of this paper, we will construct transition semi-waves, which do not have global mean speeds
(c.f. Example \ref{ex6.1}).

Now let us consider transition semi-waves for different kinds of reaction terms.  

First, we will consider the homogeneous case, i.e., $f(t,x,u)=f(u)$.
Assume that $f$ satisfies  \begin{equation}\label{f}
f(u)\in {C}^{1}([0,1]), f(0)=f(1)=0, f^{\prime}(1)<0, f(u)<0\ \text{for}\ u>1.
\end{equation} 
It is easy to see that $f$ satisfies \eqref{f} when $f\in {C}^{1}([0,1])$ is one of the following three types mentioned in \cite{DLo}:
\begin{enumerate}
	\item[($f_{M}$)] Monostable: $f(0)=f(1)=0, f^{\prime}(0)>0, f^{\prime}(1)<0, (1-u)f(u)>0\ \text{for}\ u>0, u\ne0.$
	\item[($f_{B}$)]  Bistable: $f(0)=f(\theta)=f(1)=0\ \text{for some}\ \theta\in(0,1), f^{\prime}(0)<0, f^{\prime}(1)<0, f(u)<0\ \text{for}\ u\in(0,\theta)\cup(1,+\infty), f(u)>0\ \text{for}\ u\in(\theta,1),$ and $\int_{0}^{1}f(u)du>0.$
	\item[$(f_{C})$]  Combustion: $f(u)=f(1)=0\  \text{in}\ [0,\theta]\ \text{for some}\ \theta\in(0,1), f^{\prime}(1)<0,
	f(u)<0\ \text{for}\ u\in(1,+\infty), f(u)>0\ \text{for}\ u\in(\theta,1)$, and $f$ is increasing in $(\theta,\theta+\delta_{0})$ for some $\delta_{0}>0$ small.
\end{enumerate}

\begin{defn}\label{homo-sw}
	A positive entire solution $(u(t,x),h(t))$ of \eqref{1.1} with $f=f(u)$ satisfying \eqref{f} is called a semi-wave connecting 1 and 0 if
	\begin{enumerate}
		\item[(1)] $u(t,x)$ can be written as $u(t,x)=q(h(t)-x),$ where $q\in C^{2}([0,+\infty))$,
		\item[(2)] $h^{\prime}(t)$ is a positive constant,
		\item[(3)] $\lim\limits_{x\to-\infty}u(t,x+h(t))=1$ uniformly in $t\in\mathbb{R}$.
	\end{enumerate}	
\end{defn}
It is easy to find that $q$ satisfies:
\begin{equation}\label{2.1}
\left\{
\begin{aligned}
q^{\prime\prime}-cq^{\prime}+f(q)=0\ \text{in}\ (0,+\infty),\\
q(z)>0\ \text{in}\ (0,+\infty),\\
q(0)=0, q(+\infty)=1, q^{\prime}(0)=\frac{c}{\mu},\\
\end{aligned}
\right.
\end{equation}
where $c$ is a constant equal to $h^{\prime}(t)$.
On the other hand, a solution of \eqref{2.1} gives a semi-wave $(q(ct-x),ct)$ of \eqref{1.1} with $f=f(u)$.
Note that the equation \eqref{2.1} may have no solution in general.
Our first result of this paper is as following:
\newtheorem{thm}{Theorem}[section]
\begin{thm}\label{thmhomo}
Assume that $f(t,x,u)=f(u)$ satisfies \eqref{f} and that $(c^{*},q_{c^{*}})$ is a solution of \eqref{2.1}.
Let $(u(t,x),h(t))$ be a bounded transition semi-wave of \eqref{1.1} which connects $1$ and $0$.
Then $u(t,x)=q_{c^{*}}(h(t)-x)$ and $h(t)=c^*t+h(0)$.
\end{thm}

Second, we will consider problems \eqref{time} and \eqref{space}.

For time periodic case, it is known that
\begin{equation}\label{timeode}
u_{t}=u(c(t)-u)
\end{equation} has a unique positive almost periodic solution $u_{c}(t)$ (e.g. see \cite{LLS}).

Then we can give the definition of time almost periodic semi-waves of \eqref{time}.
\begin{defn}\label{timeTW}
(\cite[Definition 2.4]{LLS}).
A positive entire solution $(u(t,x),h(t))$ of \eqref{time} is called an almost periodic semi-wave connecting $u_{c}(t)$ and $0$ if
\begin{enumerate}
	\item[(1)] $u(t,x)$ can be written as $u(t,x)=q(t,x-h(t)),$ where $q(\tau,\xi)\in C^{2}(\mathbb{R}\times(-\infty,0]))$ is almost periodic in $\tau$ uniformly with respect to $\xi\leq0$, 
	\item[(2)] $h^{\prime}(t)$ is an almost periodic function,
	\item[(3)] $\lim\limits_{x\to-\infty}u(t,x+h(t))=u_{c}(t)$ uniformly in $t\in\Real$.
\end{enumerate}	
\end{defn}

Consider
\begin{equation}\label{fixedbd}
\left\{
   \begin{aligned}
  w_{t}=w_{xx}-\mu w_{x}(t,0)w_{x}+w(c(t)-w),\ \ &t\in\Real, x<0,\\
   w(t,0)=0,\ \ &t\in\Real.\\
   \end{aligned}
   \right.
\end{equation}
It is easy to see that an almost periodic semi-wave solution of \eqref{time} induces a positive almost periodic entire solution of \eqref{fixedbd}, and vice versa.
Let $X=\{u\ \text{is continuous}: \inf\limits_{x\geq\varepsilon}u(x)>0\ \text{for any}\ \varepsilon>0, u^{\prime}(0)<0\}$.
Then for any $w_{1},w_{2}\in X$ with $w_{1}(\cdot)\leq w_{2}(\cdot)$, we can define a part metric $\rho(w_{1},w_{2})$ between $w_{1}$ and $w_{2}$ as follows:
$$\rho(w_{1},w_{2})=\inf\{\ln\alpha: \alpha\geq1, w_{2}\leq\alpha w_{1}\}.$$
With the help of the part metric, \cite{LLS} proved the following theorem:
\begin{thm}\label{thmtime}
(\cite[Theorem 2.1]{LLS}).
Assume $c(t)$ is an almost periodic function in $t\in\Real$ with
$\lim\limits_{t\to+\infty}\frac{1}{t}\int_{0}^{t}c(s)ds>0$. Then
there is a time almost periodic positive semi-wave solution $(\phi(t,x),\zeta(t))$
of \eqref{time} connecting $u_c$ and $0$. Moreover, the time almost periodic positive semi-wave solution connecting $u_c$ and $0$ is unique up to the space translation.
\end{thm}
 In this paper, we will prove the following result for \eqref{time}:
\begin{thm}\label{heterotime}
Assume that $c(t)$ is an almost periodic function in $t\in\Real$ with $\lim\limits_{t\to+\infty}\frac{1}{t}\int_{0}^{t}c(s)ds>0$.	
Let $(u(t,x),h(t))$ be a bounded transition semi-wave of \eqref{time} which connects $u_{c}(t)$ and $0$.
Then $u(t,x)$ is a time almost periodic semi-wave.
\end{thm}

For the space almost periodic case, we have the similar conclusion. First, $u_{t}=u_{xx}+u(a(x)-u)$ has a unique positive almost periodic solution $v_{a}$ (e.g. see \cite{L}).
Then we also can give the definition of space almost periodic semi-waves of \eqref{space}.
\begin{defn}\label{spaceTW}
	A positive entire solution $(u(t,x),h(t))$ of \eqref{space} is called an almost periodic semi-wave connecting $v_{a}(x)$ and $0$ if
	\begin{enumerate}
		\item[(1)] $u(t,x)$ can be written as $u(t,x)=q(h(t),x-h(t)),$ where $q(\tau,\xi)\in C^{2}(\mathbb{R}\times(-\infty,0]))$ is almost periodic in $\tau$ uniformly with respect to $\xi\leq0$, 
		\item[(2)] $h(\pm\infty)=\pm\infty$ and $g(\tau):=h^{\prime}(h^{-1}(\tau))$ is an almost periodic function,
		 i.e., $h^{\prime}(t)$ is an almost periodic function of $h(t)$.
		\item[(3)] $\lim\limits_{x\to-\infty}u(t,x+h(t))=v_{a}(x+h(t))$ uniformly in $t\in\Real$.
	\end{enumerate}	
\end{defn}
This definition is equivalent to that given in \cite[Definitions 1.1 and 1.2]{L}.

In \cite{L}, it is proved
\begin{thm}\label{thmtime}
(\cite[Theorem 1.1]{L}).
Assume $a$ is a positive almost periodic function. Then
there is a space almost periodic positive semi-wave solution $(\phi(t,x),\zeta(t))$
of \eqref{space} connecting $v_a$ and $0$.  Moreover, the time almost periodic positive semi-wave solution connecting $u_c$ and $0$ is unique up to the space translation.
\end{thm}

Our result for \eqref{space} is

\begin{thm}\label{heterospace}
Assume that $a(x)$ is a positive almost periodic function. Let $(u(t,x),h(t))$ be a bounded transition semi-wave of \eqref{space} which connects $v_{a}(x)$ and $0$.
Then $u(t,x)$ is a spatial almost periodic semi-wave.
\end{thm}

In fact, one can prove the Theorem \ref{heterotime} (resp. \ref{heterospace}) holds for more general KPP-Fisher type reaction terms, which are independent of $x$ (resp. $t$).

\section{Properties of transition semi-wave}

\subsection{Some useful known results}

In this subsection, we present some useful known results which we will need later.

\begin{lemma}\label{lem3.0}
	(\cite[Lemma 2.1]{DLo}).Suppose that $T\in(0,+\infty)$.
	Assume that $f(t,x,u)\equiv f(u)\in C^{1}$ and $f(0)=0$. Let $(u,g_{-},g_{+})$ be a positive solution of
	\begin{equation}\label{4.1}
	\left\{
	\begin{aligned}
	u_{t}=u_{xx}+f(u),\ \ &t>0,\ g_{-}(t)<x<g_{+}(t),\\
	u(t,g_{\pm}(t))=0,\ \ &t>0,\\
	g_{\pm}^{\prime}(t)=-\mu u_{x}(t,g_{\pm}(t)),\ \ &t>0,\\
	\end{aligned}
	\right.
	\end{equation}
	with initial value $ \pm g_{\pm}(0)= g_{0}>0, u(0,x)=u_{0}(x)\ \text{for}\ x\in(-g_{0},g_{0})$.
    Let $ \overline{g}_{\pm}\in C^{1}[0,T],  $ $D_{T}=\{(t,x)\in \Real^{2}: 0<t\leq T, \overline{g}_{-}(t)<x<\overline{g}_{+}(t)\},$
 $\overline{u}\in C(\overline{D_{T}})\cap C^{1,2}(D_{T})$ be positive in $D_T$
		and
	\begin{equation*}
	\left\{
	\begin{aligned}
	\overline{u}_{t}\geq\overline{u}_{xx}+f(\overline{u}),\ \ &0<t\leq T,\ \overline{g}_{-}(t)<x<\overline{g}_{+}(t),\\
	\overline{u}(t,\overline{g}_{\pm}(t))=0, \ \ &0<t\leq T,\\
	\pm \overline{g}_{\pm}^{\prime}(t)\geq\mp\mu \overline{u}_{x}(t,\overline{g}_{\pm}(t)),\ \ &0<t\leq T.\\
	\end{aligned}
	\right.
	\end{equation*}
	If $[-g_{0},g_{0}]\subset[\overline{g}_{-}(0),\overline{g}_{+}(0)]$ and $u_{0}(x)\leq\overline{u}(0,x)$ for $x\in[-g_{0},g_{0}]$, then
	$$\overline{g}_{-}(t)\leq g_{-}(t),{g}_{+}(t)\leq \overline{g}_{+}(t)\ \text{for}\ t\in(0,T],$$
	$$u(t,x)\leq\overline{u}(t,x)\ \text{for}\ t\in(0,T]\ \text{and}\ x\in({g}_{-}(t),{g}_{+}(t)).$$
\end{lemma}

\begin{lemma}\label{lem4.1}
(\cite[Lemma 2.2]{DLo}). Suppose that $T\in(0,+\infty)$.
	Assume that $f(t,x,u)\equiv f(u)\in C^{1}$ and $f(0)=0$. Let $(u,g_{-},g_{+})$ be given as in Lemma \ref{lem3.0}.
Moreover, let $\overline{g}_{\pm}\in C^{1}[0,T], D_{T}=\{(t,x)\in \Real^{2}: 0<t\leq T, \overline{g}_{-}<x<\overline{g}_{+}\}, 
\overline{u}\in C(\overline{D_{T}})\cap C^{1,2}(D_{T})$ be positive in $D_T$ and
\begin{equation*}
\left\{
   \begin{aligned}
  \overline{u}_{t}\geq\overline{u}_{xx}+f(\overline{u}),\ \ &0<t\leq T,\ \overline{g}_{-}(t)<x<\overline{g}_{+}(t),\\
   \overline{u}(t,\overline{g}_{-}(t))\geq u(t,\overline{g}_{-}(t)),\ \ &0<t\leq T,\\
   \overline{u}(t,\overline{g}_{+}(t))=0,\ \ &0<t\leq T,\\
   \overline{g}_{+}^{\prime}(t)\geq-\mu \overline{u}_{x}(t,\overline{g}_{+}(t)),\ \ &0<t\leq T.\\
   \end{aligned}
   \right.
\end{equation*}
If $\overline{g}_{-}(t)\geq g_{-}(t)$ for $t\in[0,T]$, $g_{0}\leq \overline{g}_{+}(0)$, and
$u_{0}(x)\leq \overline{u}(0,x)$ for $x\in[\overline{g}_{-}(0),g_{0}]$, then
$$g_{+}(t)\leq \overline{g}_{+}(t)\ \text{for}\ t\in(0,T],$$
$$u(t,x)\leq \overline{u}(t,x)\ \text{for}\ t\in(0,T]\ \text{and}\ x\in(\overline{g}_{-}(t),{g}_{+}(t)).$$
\end{lemma}

\begin{lemma}\label{lem4.2}
(\cite[Proposition 2.14]{DDL}). Suppose that  $T\in(0,+\infty)$.
	Assume that $f(t,x,u)\equiv f(u)\in C^{1}$, $f(0)=0$, and $f(u)\leq Ku$ for $u\geq0$ with some $K>0$. Let $(u,g)$ be a solution of
\begin{equation}\label{4.2}
\left\{
\begin{aligned}
u_{t}=u_{xx}+f(u),\ \ &0<t\leq T,\ x<g(t),\\
u(t,g(t))=0,\ \ \ &0<t\leq T,\\
g^{\prime}(t)=-\mu u_{x}(t,g(t)),\ \  &0<t\leq T,\\
\end{aligned}
\right.
\end{equation}
with initial value $g(0)=g_{0}, u(0,x)=u_{0}(x)\ \text{for}\ x\in(-\infty,g_{0})$.
Let $\overline{g}\in C([0,T])\cap C^{1}((0,T]), D_{T}=\{(t,x)\in \Real^{2}: 0<t\leq T, x\leq\overline{g}\}, \overline{u}\in C(\overline{D_{T}})\cap C^{1,2}(D_{T})$
 be positive and
\begin{equation*}
\left\{
   \begin{aligned}
  \overline{u}_{t}\geq\overline{u}_{xx}+f(\overline{u}),\ \ &0<t\leq T,\ x<\overline{g}(t),\\
   \overline{u}(t,\overline{g}(t))=0,\ \ &0<t\leq T,\\
  \overline{g}^{\prime}(t)\geq-\mu \overline{u}_{x}(t,\overline{g}(t)),\ \ &0<t\leq T.\\
   \end{aligned}
   \right.
\end{equation*}
If $g(0)\leq \overline{g}(0)$ and $u(0,x)\geq \overline{u}_{0}(x)$ for $x\in(-\infty,g_{0}]$, then
$$g(t)\leq \overline{g}(t)\ \text{for}\ t\in(0,T],$$
$$u(t,x)\leq \overline{u}(t,x)\ \text{for}\ t\in(0,T]\ \text{and}\ x\in(-\infty,{g}(t)).$$
\end{lemma}

\newtheorem{rem}{Remark}[section]
\begin{rem}\label{re4.1}
The triple $(\overline{u},\overline{g}_{-},\overline{g}_{+})$ in Lemma \ref{lem4.1} (resp. the
pair $(\overline{u}, \overline{g})$ in Lemma \ref{lem4.2}) is called an upper solution of \eqref{4.1}
(resp. \eqref{4.2}). We can define a lower solution through replacing the signs $\geq$ by signs $\leq$ .
Moreover, the corresponding comparison results still hold for lower solutions.
\end{rem}

\begin{rem}\label{re4.2}
A simple corollary of Lemma \ref{lem4.2} is that the solution $u(t,x)$ of \eqref{4.2} is decreasing in $x$ if $u_{0}$ is decreasing in $x$.
\end{rem}

\subsection{Properties of transition semi-wave}
In this subsection, we will always regard $p$ as a positive entire solution of $u_{t}=u_{xx}+f(t,x,u)$
and will also prove some basic properties of the transition semi-wave under some suitable assumptions. 

\newtheorem{prop}{Proposition}[section]
\begin{prop}\label{prop_prioriestimate}
Assume that $(u(t,x),h(t))$ is a positive, bounded entire solution of \eqref{1.1}. Then
\begin{equation}\label{prioriestimate}
\|w\|_{C^{1+\alpha/2,2+\alpha}(\mathbb{R}\times(-\infty,0])}+\|h^{\prime}\|_{C^{\alpha/2}(\Real)}\leq C,
\end{equation}
where $w(t,x)=u(t,x+h(t))$ and $C$ is a positive constant depending on $f, \|u\|_{\infty},$ and $L$.

In addition, suppose $\inf\limits_{(t,x)\in\Omega_{h}}p(t,x)>0$.  If $(u(t,x),h(t))$ is a transition semi-wave which connects $p$ and $0$, then
$\inf\limits_{t\in \Real }h^{\prime}(t)>0$.
\end{prop}
\begin{proof}
Let $h_{k}(t)=h(t+k)-h(k), u_{k}(t,x)=u(t+k,x+h(k))$, where $k\in\mathbf{Z}.$
Then $(u_{k},h_{k})$ satisfies
\begin{equation*}
\left\{
   \begin{aligned}
  (u_{k})_{t}=(u_{k})_{xx}+f(t+k,x+h(k),u_{k}),\ \ &t\in\Real, x<h_{k}(t),\\
   u_{k}(t,h_{k}(t))=0,\ h_{k}^{\prime}(t)=-\mu (u_{k})_{x}(t,h(t)),\ &t\in\Real.\\
   \end{aligned}
   \right.
\end{equation*}
By \cite[Theorem 2,11]{DDL}, we have $\|h_{k}\|_{C^{1+\alpha/2}([1,3])}\leq C_{1}$, where $C_{1}$
is a positive constant depending on $f$ and $\|u\|_{\infty}$ but not depending on $k$. Hence
$\|h^{\prime}\|_{C^{\alpha/2}(\Real)}\leq C_{1}$.

Let $w(t,x)=u(t,x+h(t))$ for any $t\in\Real, x\leq0$. Then $w$ satisfies
\begin{equation*}
\left\{
   \begin{aligned}
  w_{t}=w_{xx}+h^{\prime}(t)w_{x}+f(t,x+h(t),w),\ \ &t\in\Real, x<0,\\
   w(t,0)=0,\ h^{\prime}(t)=-\mu w_{x}(t,0),\ &t\in\Real.\\
   \end{aligned}
   \right.
\end{equation*}
Since $h^{\prime}\in{C^{\alpha/2}(\Real)}$ and $f(\cdot,\cdot,w)\in C^{\alpha/2,\alpha}(\Real^{2})$
uniformly for $w\in[0,\|u\|_{\infty}]$, the parabolic Schauder estimates yield that
$$\|w\|_{C^{1+\alpha/2,2+\alpha}([t,t+1]\times[-2,0])}\leq C_{2},$$
$$\|w\|_{C^{1+\alpha/2,2+\alpha}([t,t+1]\times[-(n+1),-n])}\leq C_{2},\ n=1,2,\cdots,$$
where $C_{2}$ is independent of $t$.
Therefore, $\|w\|_{C^{1+\alpha/2,2+\alpha}(\Real\times(-\infty,0])}\leq C_{2}.$
Thus \eqref{prioriestimate} holds.

Next we will show that $\inf\limits_{t\in \Real}h^{\prime}(t)>0$ provided $(u(t,x),h(t))$ is a transition
semi-wave which connects $p$ and $0$ with $\inf\limits_{(t,x)\in\Omega_{h}}p(t,x)>0$.
Suppose that $\inf\limits_{t\in \Real}h^{\prime}(t)=0$. Then there exists $\{t_{n}\}_{n\in\mathbb{N}}$ such that
$\lim\limits_{n\to\infty}h^{\prime}(t_{n})=0$.
Let $h_{n}(t)=h(t+t_{n})-h(t_{n}), u_{n}(t,x)=u(t+t_{n},x+h(t_{n}))$.
Then $(u_{n},h_{n})$ satisfies
\begin{equation*}
\left\{
   \begin{aligned}
  (u_{n})_{t}=(u_{n})_{xx}+f_{n}(t,x,u_{n}),\ \ &t\in\Real, x<h_{n}(t),\\
   u_{n}(t,h_{n}(t))=0,\ h_{n}^{\prime}(t)=-\mu (u_{n})_{x}(t,h_{n}(t)),\ &t\in\Real,\\
   \end{aligned}
   \right.
\end{equation*}
where $f_{n}(t,x,s)=f(t+t_{n},x+h(t_{n}),s)$. By the priori estimates,
we can find some subsequence of $\{(u_{n},h_{n},f_{n})\}_{n\in\mathbb{N}}$, still denoted by $\{(u_{n},h_{n},f_{n})\}_{n\in\mathbb{N}}$,
$h_{\infty}\in C^{1}_{loc}(\Real)$,
$u_{\infty}\in C^{1,2}_{loc}(\{(t,x): t\in\Real,x\leq h_{\infty}(t)\})$,
and $f_{\infty}(\cdot,\cdot,s)\in C^{\beta/2,\beta}_{loc}(\Real^{2})$ locally in $s\in\Real$ with some
$\beta<\alpha$ such that
$$h_{n}\to h_{\infty}\ \text{in}\ C^{1}_{loc}(\Real),\ u_{n}\to u_{\infty}\ \text{in}\ C^{1,2}_{loc}(\{(t,x): t\in\Real,x\leq h_{\infty}(t)\}),$$
$$f_{n}(\cdot,\cdot,s)\to f_{\infty}(\cdot,\cdot,s)\ \text{in}\ C^{\beta/2,\beta}_{loc}(\Real^{2})\ \text{uniformly w.r.t.}\ s\in[0,\|u\|_{\infty}].$$
Moreover, $(u_{\infty},h_{\infty})$ satisfies
\begin{equation*}
\left\{
   \begin{aligned}
  (u_{\infty})_{t}=(u_{\infty})_{xx}+f_{\infty}(t,x,u_{\infty}),\ \ &x<h_{\infty}(t),\\
   u_{\infty}(t,h_{\infty}(t))=0,\ \ &t\in\Real,\\
   h_{\infty}^{\prime}(t)=-\mu (u_{\infty})_{x}(t,h_{\infty}(t)),\ \ &t\in\Real,\\
   \end{aligned}
   \right.
\end{equation*}
with $h_{\infty}(0)=0$ and $(u_{\infty})_{x}(0,0)=0$.

On the other hand, by the definition of transition semi-waves, we can find $B>0$ larger enough
such that $u(t,h(t)-B)>\frac{1}{2}\inf\limits_{(t,x)\in\Omega_{h}}p(t,x)>0$ for $t\in\Real$.
Hence $u_{\infty}(t,h_{\infty}(t)-B)\geq\frac{1}{2}\inf\limits_{(t,x)\in\Omega_{h}}p(t,x)>0$.
Then the strong maximum principle yields that $u_{\infty}(t,x)>0$ for $t\in\Real,x< h_{\infty}(t)$.
Therefore, $(u_{\infty})_{x}(0,0)<0$ because of the Hopf's Lemma, which deduces a contradiction.
\end{proof}

\begin{prop}\label{prop3.1} Suppose  $p$ is bounded.
Let  $(u(t,x),h(t))$ be a bounded transition semi-wave of \eqref{1.1} which connects $p$ and $0$.
Assume that $\inf\limits_{(t,x)\in\Omega_{h}}p(t,x)>0$, and that $u\to f(t,x,u)$ is decreasing in $[p(t,x),+\infty)$ for all $(t,x)\in\Real^{2}$.
Then $u(t,x)<p(t,x)$ for any $(t,x)\in\Omega_{h}.$
\end{prop}
\begin{proof}
The strategy of the proof is similar to that of \cite[Lemma 4.3]{BH12}.
Note that $m:=\inf\limits_{(t,x)\in\Omega_{h}}\{p(t,x)-u(t,x)\}$ is well defined since $u$ is bounded and $\inf\limits_{(t,x)\in\Omega_{h}}p(t,x)>0$.
Suppose that $m<0$.
Then there exists a sequence $\{(t_{n},x_{n})\}_{n\in\mathbb{N}}$ in $\Omega_{h}$ such that
\begin{equation}\label{3.1}
p(t_{n},x_{n})-u(t_{n},x_{n})\to m\ \text{as}\ n\to\infty.
\end{equation}
Claim: $\{x_{n}-h(t_{n})\}_{n\in\mathbb{N}}$ is bounded.\\
Proof of Claim: If not, then we must have $x_{n_{k}}-h(t_{n_{k}})\to -\infty$ for some subsequence of
$\{(t_{n},x_{n})\}_{n\in\mathbb{N}}$. Therefore, it follows from \eqref{defofTranW} that
$$\lim\limits_{k\to\infty}\big(p(t_{n_{k}},x_{n_{k}})-u(t_{n_{k}},x_{n_{k}})\big)=0>m,$$
which contradicts \eqref{3.1}. Hence $\{x_{n}-h(t_{n})\}_{n\in\mathbb{N}}$ is bounded.

Noting that $u(t,h(t))=0$, we can find $\kappa>0$ such that
\begin{equation}\label{3.2}
p(t,x)-u(t,x)>\inf\limits_{(t,x)\in\Omega_{h}}p(t,x)/2>0\ \text{for}\ -\kappa<x-h(t)\leq0
\end{equation}
since $u_{x}$ is uniformly continuous in $\overline{\Omega}_{h}$.
As a consequence, $x_{n}-h(t_{n})\leq-\kappa$.

It is clear that $\{x_{n}-h(t_{n}-1)\}_{n}$ is also bounded since $h^{\prime}$ is bounded.
Now take $\rho\in(0,\frac{\kappa}{4})$ such that $|h(s)-h(t)|\leq\frac{\kappa}{2}$ for any $|s-t|\leq\rho$
and $K\in\mathbb{N}$ such that
\begin{equation}\label{4.9}
K\rho\geq\max\{1, \sup\limits_{n\in\mathbb{N}}|x_{n}-h(t_{n}-1)|\}.
\end{equation}
For each $n$ and $i=0,1,\cdots,K$, set
$$x_{n,i}=x_{n}+\frac{i}{K}\big(h(t_{n}-1)-x_{n}\big),$$
and
$$E_{n,i}=[t_{n}-\frac{i+1}{K},t_{n}-\frac{i}{K}]\times[x_{n,i}-2\rho,x_{n,i}+2\rho].$$
Then $|x_{n,i+1}-x_{n,i}|\leq\rho$ for $i=0,1,\cdots,K-1$ by \eqref{4.9}.
Consider $i=0$. For $t\in[t_{n}-\frac{1}{K},t_{n}],$ we have
$$x_{n,0}+2\rho-h(t)=x_{n}-h(t_{n})+2\rho+h(t_{n})-h(t)\leq2\rho-\kappa+h(t_{n})-h(t)<0.$$
Hence $E_{n,0}\subset\{(t,x): x-h(t)<0\}$. Let $w=p-m-u$.
Then $w\geq0$ in $E_{n,0}$.
Moreover, $p-m$ satisfies
$$(p-m)_{t}=(p-m)_{xx}+f(t,x,p)
\geq (p-m)_{xx}+f(t,x,p-m).$$
Therefore,
$$w_{t}\geq w_{xx}+\frac{f(t,x,p-m)-f(t,x,u)}{p-m-u}w$$
in $\{(t,x): x-h(t)<0\}$ with $\lim\limits_{n\to\infty}w(t_{n},x_{n})=0$.
Then the linear parabolic estimates imply that $\lim\limits_{n\to\infty}w(t_{n}-\frac{1}{K},x_{n,1})=0$, i.e.,
$\lim\limits_{n\to\infty}p(t_{n}-\frac{1}{K},x_{n,1})-m-u(t_{n}-\frac{1}{K},x_{n,1})=0.$
This and \eqref{3.2} yield that $x_{n,1}-h(t_{n}-\frac{1}{K})\leq-\kappa$ for $n$ large.
From this, we have $x_{n,1}+2\rho-h(t)<0$ for any $t\in[t_{n}-\frac{2}{K},t_{n}-\frac{1}{K}]$.
Hence $E_{n,1}\subset\Omega_{h}$. Repeat the arguments above, and finally, by induction, we have
$$x_{n,i}-h(t_{n}-\frac{i}{K})\leq-\kappa,\ i=0,1,\cdots,K,$$
which contradicts $x_{n,K}=h(t_{n}-1)$. Hence $m\geq0$, i.e., $p(t,x)\geq u(t,x)$ in $\Omega_{h}$.

If $p(t_{0},x_{0})=u(t_{0},x_{0})$ for some $(t_{0},x_{0})\in\Omega_{h}$,
then the strong parabolic maximum principle implies that $p(t,x)\equiv u(t,x)$ in $\Omega_{h}$,
which contradicts \eqref{3.2}.
\end{proof}

\begin{prop}\label{prop3.2}Suppose that $p$ is bounded.
Let $(u(t,x),h(t))$ be a bounded transition semi-wave of \eqref{1.1} which connects $p$ and $0$. Assume that $\inf\limits_{(t,x)\in\Omega_{h}}p(t,x)>0$.
Then for any fixed $a>0$, $\inf\limits_{x-h(t)\leq-a}u(t,x)>0$.
\end{prop}
\begin{proof}
Noting that $0<\inf\limits_{t\in\mathbb{R}}h^{\prime}(t)\leq\sup\limits_{t\in\mathbb{R}}h^{\prime}(t)<+\infty$ and $u_{x}$ is uniformly continuous in $\overline{\Omega}_{h}$,
we always have $\inf\limits_{-2\kappa\leq x-h(t)\leq-\kappa}u(t,x)>0$ for $\kappa>0$ small.
Then it is sufficient to show $\inf\limits_{x-h(t)\leq-2\kappa}u(t,x)>0$.
If not, then there exists $\{(t_{n},x_{n})\}_{n\in\mathbb{N}}$ with $x_{n}-h(t_{n})<-2\kappa$ such that
$u(t_{n},x_{n})\to0$ as $n\to\infty$.
Furthermore, $\{x_{n}-h(t_{n})\}_{n\in\mathbb{N}}$ is a bounded sequence. In fact, there exists a subsequence $\{(t_{n_{k}},x_{n_{k}})\}_{k\in\mathbb{N}}$
of $\{(t_{n},x_{n})\}_{n\in\mathbb{N}}$ such that $x_{n_k}-h(t_{n_k})\to-\infty$ if $\{x_{n}-h(t_{n})\}_{n\in\mathbb{N}}$ is unbounded.
Then the definition of the transition wave yields that $\lim\limits_{k\to\infty}\big|p(t_{n_{k}},x_{n_{k}})-u(t_{n_{k}},x_{n_{k}})\big|
=\lim\limits_{k\to\infty}\big|p(t_{n_{k}},x_{n_{k}})\big|\geq\inf\limits_{(t,x)\in\Omega_{h}}p(t,x)>0$, which contradicts \eqref{defofTranW}. 
Hence we may assume that $-A\leq x_{n}-h(t_{n})<-2\kappa$.
Take $K\in\mathbb{Z}^{+}$ with
$\kappa K>A+b_{0},$ where $b_{0}=\sup\limits_{t\in\mathbb{R}}|h^{\prime}(t)|$. For $i=0,1,\cdots,K-1$, we set
$$t_{n}^{i}=t_{n}-\frac{i}{K},\ x_{n}^{i}=x_{n}+\frac{A}{K}i,$$
$$E_{n}^{i}=[t_{n}^{i+1},t_{n}^{i}]\times[x_{n}^{i},x_{n}^{i+1}].$$
Consider $E_{n}^{0}$. For $t\in[t_{n}^{1},t_{n}^{0}]$,
$$x_{n}^{1}-h(t)=
x_{n}^{0}+\frac{A}{K}-h(t_{n}^{0})+h(t_{n}^{0})-b(t)<-2\kappa+\frac{A+b_{0}}{K}\leq-\kappa,$$
i.e., $E_{n}^{0}\subset\{(t,x): x-h(t)\leq-\kappa\}.$ Moreover, we have either
$$-\kappa\geq x_{n}^{1}-h(t_{n}^{1})\geq-2\kappa$$
or
$$x_{n}^{1}-h(t_{n}^{1})<-2\kappa.$$
If $x_{n}^{1}-h(t_{n}^{1})<-2\kappa$ holds, then we have $E_{n}^{1}\subset\{(t,x): x-h(t)\leq-\kappa\}.$
In fact, for $t\in[t_{n}^{2},t_{n}^{1}]$,
$$x_{n}^{2}-h(t)=
x_{n}^{1}+\frac{A}{K}-h(t_{n}^{1})+h(t_{n}^{1})-h(t)<-2\kappa+\frac{A+b_{0}}{K}\leq-\kappa.$$
As before, we have either
$$-\kappa\geq x_{n}^{2}-h(t_{n}^{2})\geq-2\kappa$$
or
$$x_{n}^{2}-h(t_{n}^{2})<-2\kappa.$$
By induction, for any $n$, there exists $k_{n}\in\{1,2,\cdots,K-1\}$ such that
\begin{equation}\label{2.2}
\left\{
   \begin{aligned}
 -\kappa\geq x_{n}^{k_{n}}-h(t_{n}^{k_{n}})\geq-2\kappa,\\
 x_{n}^{i}-h(t_{n}^{i})<-2\kappa, E_{n}^{i}\subset\Omega_{h}(\kappa),\\
   \end{aligned}
   \right.
\end{equation}
for $i<k_{n}$, where $\Omega_{h}(\kappa):=\{(t,x): x-h(t)\leq-\kappa\}$.
Since $x_{n}^{K}-h(t_{n}^{K})\geq0$, we have $k_{n}\neq K$.
Up to extraction of some subsequence, we can assume that $k_{n}\equiv k$.
Applying the linear parabolic estimates to $u$, we have
$u(t_{n}^{i},x_{n}^{i})\to 0$ as $n\to\infty$ for $i=1,2,\cdots,k.$
On the other hand, $u(t_{n}^{k},x_{n}^{k})\geq\inf\limits_{-2\kappa\leq x-h(t)\leq-\kappa}u(t,x)>0$
since \eqref{2.2}, which contradicts $\lim\limits_{n\to\infty}u(t_{n}^{k},x_{n}^{k})=0.$
\end{proof}

\begin{prop}\label{propspaceshift}Suppose that $p$ is bounded.
Let $(u(t,x),h(t))$ be a bounded transition semi-wave of \eqref{1.1} which connects $p$ and $0$ with $\inf\limits_{(t,x)\in\Omega_{h}}p(t,x)>0$.
Assume that there is $\theta>0$ such that $u\to f(t,x,u)$ is decreasing
in $[p(t,x)-\theta,+\infty)$ for all $(t,x)\in\Real^{2}$.
If $p(t,x)$ and $f(t,x,u)$ are decreasing in $x$, then $u(t,x)$ is decreasing in $x$. Specially, if $p(t,x)$ and $f(t,x,u)$ are independent of $x$, then $u(t,x)$ is decreasing in $x$. 
\end{prop}
\begin{proof}
We prove this proposition in two steps. Denote $u^{\xi}(t,x)=u(t,x-\xi)$.\\
Step 1: Show that there exists some constant $B>0$ such that for any $\xi\geq B,$
\begin{equation}\label{spaceshift}
u^{\xi}(t,x)\geq u(t,x)\ \text{for}\ (t,x)\in\Omega_{h}.
\end{equation}

Since $(u(t,x),h(t))$ is a transition semi-wave of \eqref{1.1} which connects $p$ and $0$,
we can find $B>0$ such that $u(t,x)>p(t,x)-\frac{\theta}{2}$ for any $x-h(t)<-B$.
Note that $u(t,x)$ is bounded. Then
$$\varepsilon^{*}=\inf\{\varepsilon>0: u^{\xi}(t,x)+\varepsilon\geq u(t,x)\ \text{for}\ (t,x)\in\Omega_{h}\}$$
is well defined. In particular,
$$u^{\xi}(t,x)+\varepsilon^{*}\geq u(t,x)\ \text{for}\ (t,x)\in\Omega_{h}.$$
It is sufficient to show that $\varepsilon^{*}=0$.

Suppose that $\varepsilon^{*}>0$. Then there exist sequences $\{\varepsilon_{n}\}_{n\in\mathbb{N}}$ increasing to $\varepsilon^{*}$
and $\{(t_{n},x_{n})\}_{n\in\mathbb{N}}$ with $x_{n}-h(t_{n})<0$ such that
\begin{equation}\label{3.7}
u^{\xi}(t_{n},x_{n})+\varepsilon_{n}<u(t_{n},x_{n}).
\end{equation}
Since $0<\sup\limits_{t\in\mathbb{R}}h^{\prime}(t)<+\infty$
and $u_{x}(t,x)$ is uniformly continuous in $\Omega_{h}$,
we can find some constant $\kappa>0$ such that
\begin{equation}\label{3.8}
u^{\xi}(t,x)+\frac{1}{2}\varepsilon^{*}>u(t,x)\ \forall t\in\Real, h(t)-\kappa<x<h(t).
\end{equation}
We may assume that $\varepsilon_{n}>\frac{1}{2}\varepsilon^{*}$.
It follows from \eqref{3.7} and \eqref{3.8} that $x_{n}-h(t_{n})\leq-\kappa$.\\
Claim: $\{x_{n}-h(t_{n})\}_{n\in\mathbb{N}}$ is bounded.\\
Proof of Claim: If not, then we must have $x_{n_{k}}-h(t_{n_{k}})\to -\infty$ for some subsequence $\{(t_{n_{k}},x_{n_{k}})\}_{n\in\mathbb{N}}$ of
$\{(t_{n},x_{n})\}_{n\in\mathbb{N}}$. Therefore, it follows from \eqref{3.7} and the monotonicity of $p$ that
\begin{equation*}
\begin{split}
\lim\limits_{k\to\infty}\big(p(t_{n_{k}},x_{n_{k}}-\xi)-u(t_{n_{k}},x_{n_{k}}-\xi)\big)
&\geq\lim\limits_{k\to\infty}\big(p(t_{n_{k}},x_{n_{k}}-\xi)-
 u(t_{n_{k}},x_{n_{k}})+\varepsilon_{n_{k}}\big)\\
&\geq\lim\limits_{k\to\infty}\big(p(t_{n_{k}},x_{n_{k}})-
 u(t_{n_{k}},x_{n_{k}})\big)+\varepsilon^{*}\\
&=\varepsilon^{*}.\\
\end{split}
\end{equation*}
But
$\lim\limits_{k\to\infty}\big(p(t_{n_{k}},x_{n_{k}}-\xi)-u(t_{n_{k}},x_{n_{k}}-\xi)\big)
=0<\varepsilon^{*},$ which is a contradiction. Hence $\{x_{n}-h(t_{n})\}_{n\in\mathbb{N}}$ is bounded.

Now take the same notations $\rho, K\in\mathbb{N}, x_{n,i},$ and $E_{n,i}$ as defined in the proof of Proposition \ref{prop3.1}.
Let $w=u^{\xi}+\varepsilon^{*}-u$.
Then $w\geq0$ in $E_{n,0}$. Note that $u^{\xi}(t,x)>p(t,x)-\theta$ in $\Omega_{h}$ since $\xi\geq B$.
The monotonicity of $f$ implies that
$$(u^{\xi}+\varepsilon^{*})_{t}=(u^{\xi}+\varepsilon^{*})_{xx}+f(t,x-\xi,u^{\xi})
\geq (u^{\xi}+\varepsilon^{*})_{xx}+f(t,x,u^{\xi}+\varepsilon^{*}).$$
Therefore,
$$w_{t}\geq w_{xx}+\frac{f(t,x,u^{\xi}+\varepsilon^{*})-f(t,x,u)}{u^{\xi}+\varepsilon^{*}-u}w$$
in $E_{n,0}$ with $\lim\limits_{n\to\infty}w(t_{n},x_{n})=0$ since
$0\leq w(t_{n},x_{n})=u^{\xi}(t_{n},x_{n})+\varepsilon^{*}-u(t_{n},x_{n})\leq\varepsilon^{*}-\varepsilon_{n}.$
The same arguments as used in the proof of Proposition \ref{prop3.1} imply that
$$x_{n,i}-h(t_{n}-\frac{i}{K})\leq-\kappa,\ i=0,1,\cdots,K,$$
which contradicts $x_{n,K}=h(t_{n}-1)$. Hence $\varepsilon^{*}=0$.
That is, for any $\xi\geq B$,
$$u^{\xi}(t,x)\geq u(t,x)\ \text{for}\ (t,x)\in\Omega_{h}.$$

Now let us define
$$\xi^{*}=\inf\{\xi>0: u^{\xi^{\prime}}(t,x)\geq u(t,x)\ \forall t\in\Real, x\leq h(t),
\xi^{\prime}\geq\xi\}.$$
Then $\xi^{*}\in[0,B]$, $u^{\xi^{*}}(t,x)\geq u(t,x)$ for any $(t,x)\in\Omega_{h}$.\\
Step 2: Show that $\xi^{*}=0$.\\
If $\xi^{*}>0$, then by Proposition \ref{prop3.2} we have
$\inf\limits_{x\leq h(t)}u^{\xi^{*}}(t,x)=\inf\limits_{x\leq h(t)-\xi^{*}}u(t,x)>0=u(t,x).$
Furthermore, there exists $\kappa>0$ such that
\begin{equation}\label{3.3}
\inf\limits_{-\kappa<x-h(t)\leq0}\{u^{\xi^{*}}(t,x)-u(t,x)\}>0.
\end{equation}
Claim: $\inf\limits_{-B<x-h(t)\leq0}\{u^{\xi^{*}}(t,x)-u(t,x)\}>0.$\\
Proof of Claim: If not, then there exists $\{(t_{n},x_{n})\}_{n\in\mathbb{N}}$ with $x_{n}-h(t_{n})\in[-B,-\kappa]$
such that $\lim\limits_{n\to\infty}u^{\xi^{*}}(t_{n},x_{n})-u(t_{n},x_{n})=0.$
Let $\tilde{w}=u^{\xi^{*}}-u$. Then the monotonicity of $f$ yields that
\begin{equation*}
\tilde{w}_{t}\geq \tilde{w}_{xx}+\frac{f(t,x,u^{\xi^{*}})-f(t,x,u)}{u^{\xi^{*}}-u}\tilde{w}\ \text{in}\ \Omega_{h}.
\end{equation*}
Moreover, $\tilde{w}(t,x)\geq0$ in $\Omega_{h}$, and $\tilde{w}(t_{n},x_{n})\to 0$ as $n\to\infty$.
We can obtain a contradiction by
using a proof similar to the one in the proof of Proposition \ref{prop3.1}.
The proof of claim is complete.

By the claim in this step, we can find a constant $\xi_{0}>0$ such that
\begin{equation}\label{3.4}
\inf\limits_{-B<x-h(t)\leq0}\{u^{\xi^{*}-\xi}(t,x)-u(t,x)\}>0\ \forall \xi\in[0,\xi_{0}].
\end{equation}
Note that $u^{\xi^*}(t,x)\geq u(t,x)>p(t,x)-\frac{\theta}{2}$ for $x-h(t)\leq-B$.
Then for $\xi_{0}$ small, $u^{\xi^*-\xi}(t,x)>p(t,x)-\theta$ for any $\xi\in[0,\xi_{0}]$,
$x-h(t)\leq-B$ since $u_{x}$ is uniformly bounded. Setting
$$\tilde{\varepsilon}^{*}=\inf\{\varepsilon>0: u^{\xi^*-\xi}(t,x)+\varepsilon\geq u(t,x)\ \forall\ t\in\Real, x-h(t)\leq-B\},$$
we can still prove that $\tilde{\varepsilon}^{*}=0$ as we did in Step 1.
This and \eqref{3.4} imply that
$u^{\xi^*-\xi}(t,x)\geq u(t,x)$ in $\Omega_{h}$ for any $\xi\in[0,\xi_{0}]$,
which contradicts the definition of $\xi^*$. Hence $\xi^*=0$, i.e., for any $\xi\geq0$,
$$u^{\xi}(t,x)\geq u(t,x)\ \text{for}\ (t,x)\in\Omega_{h}.$$
Hence $u(t,x)$ is decreasing in $x$.
\end{proof}

\begin{prop}\label{proptimeshift}Suppose that $p$ is bounded.
Let $(u(t,x),h(t))$ be a bounded transition semi-wave of \eqref{1.1} which connects $p$ and $0$ with $\inf\limits_{(t,x)\in\Omega_{h}}p(t,x)>0$.
Assume that there is $\theta>0$ such that $u\to f(t,x,u)$ is decreasing
in $[p(t,x)-\theta,+\infty)$ for all $(t,x)\in\Real^{2}$.
If $p(t,x)$ and $f(t,x,u)$ are increasing in $t$, then $u(t,x)$ is increasing in $t$. Specially, if $p(t,x)$ and $f(t,x,u)$ are independent of $t$, then $u(t,x)$ is increasing in $t$.
\end{prop}
\begin{proof}
The proof is similar to that of Proposition \ref{propspaceshift}. So we only provide the outline of the proof.
Denote $u^{\tau}(t,x)=u(t+\tau,x)$. In the first step, we show that
there is $T>0$ such that for any $\tau\geq T,$
\begin{equation}\label{timeshift}
u^{\tau}(t,x)\geq u(t,x)\ \text{for}\  (t,x)\in\Omega_{h}.
\end{equation}
Then we define
$$\tau^{*}=\inf\{\tau>0: u^{\tau^{\prime}}(t,x)\geq u(t,x)\ \forall t\in\Real, x\leq h(t),
\tau^{\prime}\geq\tau\},$$
and show that $\tau^{*}=0$ in the second step.
\end{proof}

\begin{rem}\label{re3.1}
In Proposition \ref{prop3.1}, we do not need the requirement that $p$ is bounded
if we assume that $\sup\limits_{(t,x,u)\in\mathbf{R^{3}}} \big|f_{u}^{\prime}(t,x,u)\big|<+\infty$.
We do not need $\inf\limits_{t\in\mathbb{R}}h^{\prime}(t)>0$ in the Step 1 of Proposition \ref{propspaceshift},
but we need it in the first step of Proposition \ref{proptimeshift}.
We need the boundedness of $u_{x}$ in the Step 2 of Proposition \ref{propspaceshift},
while in the second step of Proposition \ref{proptimeshift} we need the boundedness of $u_{t}$.
\end{rem}

Next, we will prove the uniqueness of the global mean speed among a certain class of transition semi-waves:
\begin{thm}\label{thm3.1}Suppose that $p$ is bounded.
Let $(u,h)$ and $(\tilde{u},\tilde{h})$ be two bounded transition semi-waves of \eqref{1.1}. Both of them connect $p$ and $0$.
Suppose that $p$ and $f$ are independent of $t$ and $\inf\limits_{x\in\mathbb{R}}p(x)>0$.
We further assume that there is $\theta>0$ such that $u\to f(x,u)$ is
decreasing in $[p(x)-\theta,+\infty)$ for all $x\in\Real$,
and that both $u$ and $\tilde{u}$ have global mean speeds $c$ and $\tilde{c}$, respectively, with the stronger properties that
$$\sup\limits_{(t,s)\in\Real^2}|h(t)-h(s)-c(t-s)|<+\infty,$$
$$\sup\limits_{(t,s)\in\Real^2}|\tilde{h}(t)-\tilde{h}(s)-\tilde{c}(t-s)|<+\infty.$$
Then $c=\tilde{c}$ and there is (the smallest) $s_{*}\in\Real$ such that
$$\tilde{u}(t+s_{*},x)\geq u(t,x)\ \text{for any}\ x\leq h(t).$$
Furthermore, there exists a sequence $\{t_{n},x_{n}\}_{n\in\mathbb{N}}$ with $x_{n}-h(t_{n})$ bounded
such that
$$\tilde{u}(t_{n}+s_{*},x_{n})-u(t_{n},x_{n})\to0\ \text{as}\ n\to\infty.$$
Lastly, either $\tilde{u}(t+s_{*},x)>u(t,x)\ \text{for any}\ x\leq h(t)<\tilde{h}(t+s_{*})$
or $\tilde{u}(t+s_{*},x)=u(t,x)$ and $h(t)=\tilde{h}(t+s_{*}).$
\end{thm}
\begin{proof}
We will follow the convention that $u(t,x)=0$ for $x>h(t)$ and $\tilde{u}(t,x)=0$ for $x>\tilde{h}(t).$
First, notice that $\tilde{c}$ and $c$ are strictly positive since $\inf\limits_{t\in\mathbb{R}} h^{\prime}(t)>0$ and $\inf\limits_{t\in\mathbb{R}} \tilde{h}^{\prime}(t)>0$.
We want to show that $\tilde{c}\geq c$. We prove it in four steps. Suppose that  $\tilde{c}<c$.\\
Step 1: Let $v(t,x)=\tilde{u}(\frac{c}{\tilde{c}}t,x), g(t)=\tilde{h}(\frac{c}{\tilde{c}}t).$
Note that $c/\tilde{c}>1$, $\tilde{u}_{x}(\frac{c}{\tilde{c}}t,\tilde{h}(\frac{c}{\tilde{c}}t))>0$,
and $v_{t}(t,x)=\frac{c}{\tilde{c}}\tilde{u}_{t}(\frac{c}{\tilde{c}}t,x)\geq0$ by Proposition \ref{proptimeshift}.
Then $(v,g)$ satisfies
\begin{equation}\label{t3.1.1}
\left\{
   \begin{aligned}
  v_{t}\geq \frac{\tilde{c}}{c}v_{t}=v_{xx}+f(x,v),\ \ &t>0,\ x<g(t),\\
   v(t,g(t))=0,\ g^{\prime}(t)\geq-\mu v_{x}(t,g(t)),\ \ &t>0.\\
   \end{aligned}
   \right.
\end{equation}
Hence $(v,g)$, as well as all its time-shifts, is a upper-solution of \eqref{4.2}. Moreover,
it is easy to find that
$$\lim\limits_{x\to-\infty}|v(t,x+g(t))-p(t,x+g(t))|=0\ \text{uniformly in}\ t\in\Real,$$
and
$$\sup\limits_{(t,s)\in\Real^2}|g(t)-g(s)-c(t-s)|=
\sup\limits_{(t,s)\in\Real^2}|\tilde{h}(\frac{c}{\tilde{c}}t)-\tilde{h}(\frac{c}{\tilde{c}}s)-
\tilde{c}(\frac{c}{\tilde{c}}t-\frac{c}{\tilde{c}}s)|<+\infty.$$
Set $v^{s}(t,x)=v(t+s,x), g^{s}(t)=g(t+s).$ Then $(v^{s},g^{s})$ still satisfies \eqref{t3.1.1}.\\
Step 2: Show that there exists $s_{*}>-\infty$ such that $v^{s_{*}}(t,x)\geq u(t,x).$ Moreover, $v^{s_{*}}(t,x)>u(t,x)$ for $x<h(t).$\\
Note that $|g(t)-g(s)-(h(t)-h(s))|\leq|g(t)-g(s)-c(t-s)|+|h(t)-h(s)-c(t-s)|<+\infty$.
Then $|g(t)-h(t)|$ is bounded. We can find $B>0$ such that
$u(t,x)>p(x)-\frac{\theta}{2}$ for any $x-h(t)<-B$ and
$v(t,x)>p(x)-\frac{\theta}{2}$ for any $x-g(t)<-B$.
Taking $s_{0}>0$ large, say $s_{0}>\frac{\sup\limits_{t\in\mathbb{R}}|g(t)-h(t)|+B}{\inf\limits_{t\in\mathbb{R}} g^{\prime}(t)}$, we have
$g^{s}(t)\geq h(t)+B$ for any $s\geq s_{0}$. Therefore, $v^{s}(t,x)>p(x)-\frac{\theta}{2}$ for any $x<h(t).$
Using an argument similar to the one in the proof of Step 1 in Proposition \ref{propspaceshift},
we have $v^{s_{0}}(t,x)\geq u(t,x)$.
Set $s_{*}=\inf\{s\in\mathbb{R}: v^{\tau}(t,x)\geq u(t,x)\ \text{for}\ x\leq h(t), \tau\geq s\}$.
Then $s_{*}>-\infty$. In fact, if there exits a sequence $\{s_n\}_{n\in\mathbb{N}}$ with $s_{n}\to-\infty$ such that
$v^{s}(t,x)> u(t,x)\ \text{for}\ x\leq h(t)$, then $v(s_{n},x_{0})>u(0,x_{0})>0$ for any $x_{0}<h(0)$.
But this contradicts $v(s_{n},x_{0})=0$ for $n$ large since $\lim\limits_{n\to\infty}h(s_{n})=-\infty$. Hence $v^{s_{*}}(t,x)\geq u(t,x),$ and $g^{s_{*}}(t)\geq h(t)$.
If $v^{s_{*}}(s,y)=u(s,y)$ for some $y<h(s)$, then by strong maximum principle we have $v^{s_{*}}(t,x)\equiv u(t,x)$. Therefore,
$$\frac{\tilde{c}}{c}v^{s_{*}}_{t}=v^{s_{*}}_{xx}+f(x,v^{s_{*}}),\ x<h(t)$$
and
$$u_{t}=u_{xx}+f(x,u),\ x<h(t)$$
imply that $(1-\frac{\tilde{c}}{c})u_{t}(t,x)=0$. Hence $u_{t}(t,x)\equiv0$, i.e., $u$ is independent of $t$. That is impossible.\\
Step 3: Show that $\inf\limits_{t\in\mathbb{R}}\{g^{s_{*}}(t)-h(t)\}>0$.\\
Suppose that $\inf\limits_{t\in\mathbb{R}}\{g^{s_{*}}(t)-h(t)\}=0$. Then there are two cases we need to consider:\\
Case 1: There exists $t_{0}$ such that $g^{s_{*}}(t_{0})-h(t_{0})=0$.\\
Note that $g^{s_{*}}(t)-h(t)\geq0$ for any $t\in\Real$.
Then $(g^{s_{*}})^{\prime}(t_{0})-h^{\prime}(t_{0})=0$. Hence
\begin{equation}\label{t3.1.2}
-\mu v^{s_{*}}_{x}(t_{0},g^{s_{*}}(t_{0}))\leq (g^{s_{*}})^{\prime}(t_{0})=h^{\prime}(t_{0})=-\mu u_{x}(t_{0},h(t_{0})).
\end{equation}
On the other hand, by Step 2 and Hopf's Lemma,
we have $(v^{s_{*}}-u)_{x}(t_{0},h(t_{0}))<0$, i.e.,
$v^{s_{*}}_{x}(t_{0},g^{s_{*}}(t_{0}))<u_{x}(t_{0},h(t_{0})),$ which contradicts \eqref{t3.1.2}.
Thus Case 1 can not occur.\\
Case 2: There exists $\{t_{n}\}_{n\in\mathbb{N}}$ such that $\lim\limits_{n\to\infty}\big(g^{s_{*}}(t_{n})-h(t_{n})\big)=0$.\\
Let $$v^{s_{*}}_{n}(t,x)=v^{s_{*}}(t+t_{n},x+g^{s_{*}}(t_{n})),\ g^{s_{*}}_{n}(t)=g^{s_{*}}(t+t_{n})-g^{s_{*}}(t_{n}),$$
$$u_{n}(t,x)=u(t+t_{n},x+h(t_{n})),\ h_{n}(t)=h(t+t_{n})-h(t_{n}),$$
Then $\liminf\limits_{n\to\infty}(g^{s_{*}}_{n}(t)-h_{n}(t))\geq0$ for any $t\in\mathbb{R}$
and $\lim\limits_{n\to\infty}(g^{s_{*}}_{n}(0)-h_{n}(0))=0.$
By the priori estimates, we can find some subsequences of $\{(v^{s_{*}}_{n},g^{s_{*}}_{n})\}_{n\in\mathbb{N}}$ and $\{(u_{n},h_{n})\}_{n\in\mathbb{N}}$, still denoted by $\{(v^{s_{*}}_{n},g^{s_{*}}_{n})\}_{n\in\mathbb{N}}$ and $\{(u_{n},h_{n})\}_{n\in\mathbb{N}}$,
$g^{s_{*}}_{\infty}\in C^{1}_{loc}(\Real), h_{\infty}\in C^{1}_{loc}(\Real)$,
$v^{s_{*}}_{\infty}\in C^{1,2}_{loc}(\{(t,x): t\in\Real,x\leq g^{s_{*}}_{\infty}(t)\})$,
and $u_{\infty}\in C^{1,2}_{loc}(\{(t,x): t\in\Real,x\leq h_{\infty}(t)\})$
such that
$$g^{s_{*}}_{n}\to g^{s_{*}}_{\infty}\ \text{in}\ C^{1}_{loc}(\Real),\
v^{s_{*}}_{n}\to v^{s_{*}}_{\infty}\ \text{in}\ C^{1,2}_{loc}(\{(t,x): t\in\Real,x< g^{s_{*}}_{\infty}(t)\}),$$
$$h_{n}\to h_{\infty}\ \text{in}\ C^{1}_{loc}(\Real),\
u_{n}\to u_{\infty}\ \text{in}\ C^{1,2}_{loc}(\{(t,x): t\in\Real,x< h_{\infty}(t)\}).$$
Moreover, $v^{s_{*}}_{\infty}(t,x)\geq u_{\infty}(t,x)$, $g^{s_{*}}_{\infty}(t)-h_{\infty}(t))\geq0$ for any $t\in\mathbb{R}$
and $(g^{s_{*}}_{\infty}(0)-h_{\infty}(0))=0$.
Furthermore, there exists a subsequence of $\{f(\cdot+g^{s_{*}}(t_{n}),s)\}$, still denoted by $\{f(\cdot+g^{s_{*}}(t_{n}),s)\}$,
such that $f(\cdot+g^{s_{*}}(t_{n}),s)\to f_{\infty}(\cdot,s)$ in $C^{\alpha^{\prime}}_{loc}(\mathbb{R})$ locally in $s\in\Real$,
where $f_{\infty}(\cdot,s)\in C^{\alpha^{\prime}}_{loc}(\mathbb{R})$ with $\alpha^{\prime}<\alpha$.
The same conclusion is still valid for $f(\cdot+h(t_{n}),s)$, i.e.,
$f(\cdot+h(t_{n}),s)\to \tilde{f}_{\infty}(\cdot,s)$ in $C^{\alpha^{\prime}}_{loc}(\mathbb{R})$ locally in $s\in\Real$ for some $\tilde{f}_{\infty}(\cdot,s)\in C^{\alpha^{\prime}}_{loc}(\mathbb{R})$.
Since $\lim\limits_{n\to\infty}\big(g^{s_{*}}(t_{n})-h(t_{n})\big)=0$, we have $f_{\infty}(\cdot,s)=\tilde{f}_{\infty}(\cdot,s)$.
We also have
\begin{equation*}
\left\{
   \begin{aligned}
  (v^{s_{*}}_{\infty})_{t}\geq \frac{\tilde{c}}{c}(v^{s_{*}}_{\infty})_{t}=
  (v^{s_{*}}_{\infty})_{xx}+f_{\infty}(x,v^{s_{*}}_{\infty}),\ &t>0, x<g^{s_{*}}_{\infty}(t),\\
   v^{s_{*}}_{\infty}(t,g^{s_{*}}_{\infty}(t))=0,\ (g^{s_{*}}_{\infty})^{\prime}(t)\geq-\mu v_{x}(t,g^{s_{*}}_{\infty}(t)),\ \ &t>0,\\
   \end{aligned}
   \right.
\end{equation*}
and
\begin{equation*}
\left\{
   \begin{aligned}
  (u_{\infty})_{t}=(u_{\infty})_{xx}+f_{\infty}(x,u_{\infty}),\ \ &t>0,\ x<h_{\infty}(t),\\
   u_{\infty}(t,h_{\infty}(t))=0,\ h_{\infty}^{\prime}(t)=-\mu (u_{\infty})_{x}(t,h_{\infty}(t)),\ \ &t>0.\\
   \end{aligned}
   \right.
\end{equation*}
This is the same situation as Case 1, which can not occur either.
Hence $\inf\limits_{t\in\mathbb{R}}\{g^{s_{*}}(t)-h(t)\}>0$.\\
Step 4: End the proof by obtaining a contradiction.\\
Claim: We have $\inf\limits_{ -B\leq x-h(t)\leq0}\{v^{s_{*}}(t,x)-u(t,x)\}>0.$\\
Proof of Claim: If not, then there exists a sequence $\{(t_{n},x_{n})\}_{n\in\mathbb{N}}$ with $-B\leq x_{n}-h(t_{n})\leq0$
such that $v^{s_{*}}(t_{n},x_{n})-u(t_{n},x_{n})\to0$ as $n\to\infty$.
Note that $\inf\limits_{t\in\mathbb{R}}v^{s_{*}}(t,h(t))>0=v(t,h(t))$ since $\inf\limits_{t\in\mathbb{R}}\{g^{s_{*}}(t)-h(t)\}>0$
and $v^{s_{*}}_{xx}$ is uniformly bounded.
Then there exists $\sigma,\kappa>0$ such that
$v^{s_{*}}(t,x)-u(t,x)>\sigma$ for $-\kappa\leq x-h(t)\leq0.$
Hence we may assume $-B\leq x_{n}-h(t_{n})<-\kappa$. Let $w=v^{s_{*}}-u$.
Then $w$ satisfies
\begin{equation}\label{t3.1.3}
w_{t}\geq w_{xx}+\frac{f(x,v^{s_{*}})-f(x,u)}{v^{s_{*}}-u}w\ \text{for}\ x-h(t)<0,
\end{equation}
$w(t,x)\geq0$ for $x-h(t)\leq0$, and $w(t_{n},x_{n})\to 0$ as $n\to\infty$.
Now take the same notations $\rho, K\in\mathbb{N}, x_{n,i},$ and $E_{n,i}$ as the proof of Proposition \ref{prop3.1}.
By the same arguments, we have a contradiction as before.
Hence $\inf\limits_{-B\leq x-h(t)\leq0}\{v^{s_{*}}(t,x)-u(t,x)\}>0.$

Now by the claim above, there exists $s_{0}>0$ small such that
\begin{equation}\label{t3.1.4}
v^{s_{*}-s}(t,x)\geq u(t,x)\ \text{for any}\ s\in[0,s_{0}], -B\leq x-h(t)\leq0.
\end{equation}
Note that $v^{s_{*}}(t,x)\geq1-\frac{\theta}{2}$ for any $ x-h(t)\leq-B.$
Then for $s_{0}$ sufficiently small,
$$v^{s_{*}-s_{0}}(t,x)\geq1-\theta$$
for any $s\in[0,s_{0}], x-h(t)\leq-B$ since $v^{s_{*}}_{x}$ is uniformly bounded. Now setting
$$\varepsilon^{*}=\inf\{\varepsilon>0: v^{s_{*}-s}(t,x)+\varepsilon\geq u(t,x)\ \forall t\in\Real, x-h(t)\leq-B\},$$
we have $\varepsilon^{*}=0$ by the same arguments as used in the proof of Proposition \ref{propspaceshift}.
Then $v^{s_{*}-s}(t,x)\geq u(t,x)$ for any $s\in[0,s_{0}], x-h(t)\leq-B.$
Hence together with \eqref{t3.1.4},
we have $v^{s_{*}-s}(t,x)\geq u(t,x)$ for any $s\in[0,s_{0}], x-h(t)\leq0,$
which contradicts the definition of $s_{*}$.
Therefore, $\tilde{c}\geq c$.

The other inequality $\tilde{c}\leq c$ follows by reversing the roles of $u$ and $\tilde{u}$.
Thus $\tilde{c}=c$. Moreover, the above arguments also imply other conclusions of the theorem.
\end{proof}

Under some assumptions on $f$ and $p$, the free boundary somehow reflects the location of level set of $u$.
Namely, we have
\begin{thm}\label{thm3.2}
Let $(u,h)$ be an entire solution of \eqref{1.1}. Assume that $f(t,x,p)\equiv0$ for some positive constant $p$ and that $\{u(t,x): t\in\mathbb{R}, x<h(t)\}=(0,p)$.
Then $(u,h)$ is a transition semi-wave of \eqref{1.1} which connects $p$ and $0$ if and only if
the following hold:\\
1) $\forall \lambda\in(0,p),\ \sup\{|x-h(t)|: u(t,x)=\lambda\}<+\infty,$\\
2) $\forall C\geq0,\ \sup\{u(t,x): |x-h(t)|\leq C\}<p.$
\end{thm}
\begin{proof}
Suppose that $(u,h)$ is a transition semi-wave of \eqref{1.1} connecting $p$ and $0$.
Then 1) follows from \eqref{defofTranW} immediately.
If 2) fails, then there exists a sequence $\{(t_{n},x_{n})\}_{n\in\mathbb{N}}$ with $-C\leq x_{n}-h(t_{n})<0$ such that $u(t_{n},x_{n})\to p$ as $n\to\infty$. Consider $w=p-u$.
Using an argument similar to the one used in the proof of Proposition \ref{prop3.1}, we can obtain a contradiction.

Conversely, suppose that 1) and 2) hold. Denote
$$\underline{p}=\liminf\limits_{x-h(t)\to-\infty}u(t,x)\ \ \text{and}\ \
\overline{p}=\limsup\limits_{x-h(t)\to-\infty}u(t,x).$$
Then $0\leq\underline{p}\leq\overline{p}\leq p.$
Suppose that $\underline{p}<\overline{p}$. Then we can always find $\{(t_{n},x_{n})\}_{n\in\mathbb{N}}$
with $x_{n}-h(t_{n})\to-\infty$ such that $u(t_{n},x_{n})=(\underline{p}+\overline{p})/2$,
which contradicts 1). Hence $\underline{p}=\overline{p}$ and
$\lim\limits_{x\to-\infty}|u(t,x+h(t))-\overline{p}|=0$ uniformly w.r.t. $t\in\Real.$
It remains to prove that $\overline{p}=p$. Suppose that $\overline{p}<p$. Then there exist
$\varepsilon, C>0$ such that $u(t,x)\leq p-\varepsilon$ for $x-h(t)\leq-C.$
Combining this with 2), we have $\sup\limits_{x\leq h(t)}\{u(t,x)\}<p$,
which contradicts $\{u(t,x): t\in\mathbb{R}, x<h(t)\}=(0,p)$.
\end{proof}

\section{Proof of Theorem \ref{thmhomo}}

\subsection{Bounded for $|h(t)-c^{*}t|$}

\begin{thm}\label{thm4.1}
Assume that \eqref{f} holds, and that $(c^{*},q_{c^{*}})$ is a solution of \eqref{2.1}.
Then $q_{c^{*}}^{\prime}(x)>0$ for $x\geq0$.
Moreover, $(c^{*},q_{c^{*}})$ is unique.
\end{thm}
\begin{proof}
First we must have $c^{*}>0$ from \eqref{2.1}.
Note that \eqref{2.1} can be written in the equivalent form
\begin{equation}\label{pq}
\left\{
   \begin{aligned}
  &q^{\prime}=p,\\
  &p^{\prime}=cp-f(q).\\
   \end{aligned}
   \right.
\end{equation}
Then the solution $q_{c^{*}}$ corresponds to a trajectory $(q_{c^{*}}(x),p_{c^{*}}(x))$
of \eqref{pq} in pq-plane with $c=c^{*}$, which starts from the point $(0,\frac{c^{*}}{\mu})$ and ends at the point $(1,0)$ as $x\to +\infty$.
Then the trajectory has slope $\big(c^{*}-\sqrt{c^{*2}-4f^{\prime}(1)}\big)/2<0$.
Suppose that there exists $x_{0}>0$ such that $p_{c^{*}}(x)>0$ for $x\in[0,x_{0})$ and $p_{c^{*}}(x_{0})=0$.
Then $p_{c^{*}}^{\prime}(x_{0})\leq0$.
Suppose that $p_{c^{*}}^{\prime}(x_{0})=0$, i.e., $q_{c^{*}}^{\prime\prime}(x_{0})=0$. This and
the first equation of \eqref{2.1} yield that $f(q_{c^{*}}(x_{0}))=0$.
Hence $q\equiv q_{c^{*}}(x_{0})$ is also a solution of $q^{\prime\prime}-c^{*}q^{\prime}+f(q)=0$,
which contradicts the uniqueness of trajectory of \eqref{pq}. Therefore,
$p_{c^{*}}^{\prime}(x_{0})<0$, which yields $f(q_{c^{*}}(x_{0}))>0$. Then the trajectory $(q_{c^{*}}(x),p_{c^{*}}(x))$
has slope $-\infty$ at $(q_{c^{*}}(x_{0}),p_{c^{*}}(x_{0}))$, and it is easy to see that the trajectory is
contained in $\{(q,p): q\in[0,x_{0}], p\leq p_{c^{*}}(x), x\in[0,x_{0}]\}$. That is impossible since
$(q_{c^{*}}(x),p_{c^{*}}(x))\to (1,0)$ as $x\to +\infty$. Hence $p_{c^{*}}(x)>0$ for $x\geq0$.

Now we will show that $(c^{*},q_{c^{*}})$ is unique. The trajectory $(q_{c^{*}}(x),p_{c^{*}}(x))$
can be expressed as a function $p=P_{c^{*}}(q), q\in[0,1]$, which satisfies
$$\frac{dP_{c^{*}}}{dq}=c^{*}-\frac{f(q)}{P_{c^{*}}}\ \text{for}\ q\in(0,1), P_{c^{*}}(0)=\frac{c^{*}}{\mu}, P_{c^{*}}(1)=0.$$
Suppose that $(c,q_{c})$ is another solution of \eqref{2.1}. We may, without loss of generality,
assume that $c<c^{*}$. Then there exists a trajectory $(q_{c}(x),p_{c}(x))$ of \eqref{pq} in pq-plane,
which starts from the point $(0,\frac{c}{\mu})$ and ends at the point $(1,0)$ as $x\to +\infty$.
Moreover, the trajectory with slope $\big(c-\sqrt{c-4f^{\prime}(1)}\big)/2<0$ at $(1,0)$
can be expressed as a function $p=P_{c}(q), q\in[0,1]$, which satisfies
$$\frac{dP_{c}}{dq}=c-\frac{f(q)}{P_{c}}\ \text{for}\ q\in(0,1), P_{c}(0)=\frac{c}{\mu}, P_{c}(1)=0.$$
Note that $P_{c^{*}}(0)=\frac{c^{*}}{\mu}>\frac{c}{\mu}=P_{c}(0)$ and
$0>\frac{dP_{c^{*}}}{dq}|_{q=1}=\big(c^{*}-\sqrt{c^{*2}-4f^{\prime}(1)}\big)/2
>\big(c-\sqrt{c-4f^{\prime}(1)}\big)/2=\frac{dP_{c}}{dq}\big|_{q=1}.$
Then there exists $q_{0}\in(0,1)$ such that $P_{c^{*}}(q_{0})=P_{c}(q_{0})$ and
\begin{equation}\label{4.5}
\frac{dP_{c^{*}}}{dq}\bigg|_{q=q_{0}}\leq\frac{dP_{c}}{dq}\bigg|_{q=q_{0}}.
\end{equation}
On the other hand,
$\frac{dP_{c^{*}}}{dq}\big|_{q=q_{0}}=c^{*}-\frac{f(q_{0})}{P_{c^{*}}(q_{0})}
>c-\frac{f(q_{0})}{P_{c}(q_{0})}=\frac{dP_{c}}{dq}\big|_{q=q_{0}},$ which contradicts \eqref{4.5}.
Hence the solution of \eqref{2.1} is unique.
\end{proof}

The existence and uniqueness of the solution of \eqref{2.1} were proved in Proposition 1.9 and Theorem 6.2 in \cite{DLo} when $f$ is  of ($f_{M}$), $(f_{B})$, or $(f_{C})$ type.


\begin{lemma}\label{lem4.3}
Assume that \eqref{f} holds, that $(c^{*},q_{c^{*}})$ is the solution of \eqref{2.1}, and that
$(v(t,x),g(t))$ is a solution of \eqref{4.2} with $T=+\infty$. If the initial value $v(0,x)=v_{0}(x)$ satisfies $1\leq\liminf\limits_{x\to-\infty}v_{0}(x)\leq\limsup\limits_{x\to-\infty}v_{0}(x)<+\infty$,
then, for any $c\in(0,c^{*})$,
there exist $\delta\in(0,-f^{\prime}(1)), T^{*}>0$ and $M>0$ such that for
$t\geq T^{*}$,
\begin{equation}\label{4.3}
ct\leq g(t),
\end{equation}
\begin{equation}\label{4.4}
v(t,x)\geq1-Me^{-\delta t}\ \text{for}\ x\in[-ct,ct],
\end{equation}
\begin{equation}\label{upperbdd}
v(t,x)\leq1+Me^{-\delta t}\ \text{for}\ x\leq g(t).
\end{equation}
\end{lemma}
\begin{proof}
The proof is divided into two steps.\\
Step 1: Let $P_{c^{*}}$ be as in the previous subsection. Consider
\begin{equation}\label{P(q)}
\left\{
   \begin{aligned}
  &\frac{dP}{dq}=c-\frac{f(q)}{P},\ q>0,\\
  &P(0)=\frac{c^{*}}{\mu}.\\
   \end{aligned}
   \right.
\end{equation}
Since $c<c^{*}$, we easily see that the unique solution $P^{c}(q)$ of this problem stays below
$P_{c^{*}}(q)$ as $q$ increases from $0$.
Therefore there exists some $Q^{c}\in(0,1]$ such that $P^{c}(q)>0$ in $[0,Q^{c})$ and $P^{c}(Q^{c})=0$.
We must have $Q^{c}<1$. If not, then $P^{c}(q)$ corresponds to a trajectory $(q^{c}(x),p^{c}(x))$ of \eqref{pq} in pq-plane,
which starts from the point $(0,\frac{c^{*}}{\mu})$ and ends at the point $(1,0)$ as $x\to +\infty$.
Hence the trajectory $(q^{c}(x),p^{c}(x))$ has slope $\frac{dP^{c}}{dq}\big|_{q=1}=\big(c-\sqrt{c-4f^{\prime}(1)}\big)/2$ at $(1,0)$.
Obviously, $\frac{dP_{c^{*}}}{dq}|_{q=1}=\big(c^{*}-\sqrt{c^{*2}-4f^{\prime}(1)}\big)/2>\frac{dP^{c}}{dq}\big|_{q=1}.$
Note also that $\frac{dP_{c^{*}}}{dq}|_{q=0}<\frac{dP^{c}}{dq}\big|_{q=0}.$
Then there exists $q_{0}\in(0,1)$ such that $P_{c^{*}}(q_{0})=P^{c}(q_{0})$ and
\begin{equation}\label{diffless}
\frac{dP_{c^{*}}}{dq}\bigg|_{q=q_{0}}\leq\frac{dP^{c}}{dq}\bigg|_{q=q_{0}}.
\end{equation}
On the other hand,
$\frac{dP_{c^{*}}}{dq}\big|_{q=q_{0}}=c^{*}-\frac{f(q_{0})}{P_{c^{*}}(q_{0})}
>c-\frac{f(q_{0})}{P^{c}(q_{0})}=\frac{dP^{c}}{dq}\big|_{q=q_{0}},$ which contradicts \eqref{diffless}.
Hence $Q^{c}<1$. It is also easily seen that, as $c$ increases to $c^{*}$,
$Q^{c}$ increases to $1$ and $P^{c}(q)\to P_{c^{*}}(q)$ uniformly, in the sense that
$\|P^{c}-P_{c^{*}}\|_{L^{\infty}([0,Q^{c}])}$. Let $x^{c}>0$ such that $q^{c}(x^{c})=Q^{c}$.
For $t\geq0$, we define $k_{c}(t)=x^{c}+ct$ and
\begin{equation*}
w_{c}(t,x)=\left\{
   \begin{aligned}
  q^{c}(k(t)-x),\ \ &x\in[ct,k_{c}(t)],\\
  q^{c}(x^{c}),\ \ &x\in[-ct,ct],\\
  q^{c}(k(t)+x),\ \ &x\in[-k_{c}(t),-ct].\\
   \end{aligned}
   \right.
\end{equation*}
Step 2: Fix $\hat{c}\in(c,c^{*})$. Let $(u,g_{-},g_{+})$ be a solution of \eqref{4.1} with initial value $\pm g_{\pm}(0)=k_{\hat{c}}(0)$ and $u(0,x)=w_{\hat{c}}(0,x)$.
One can easily check that $(w_{\hat{c}}(t,x),-k_{\hat{c}}(t),k_{\hat{c}}(t))$ is a lower solution of \eqref{4.1} for $t\geq 0$.
Hence by Lemma \ref{lem3.0} for lower solution version, we have
\begin{equation}\label{compare1}
\left\{
\begin{aligned}
 {g}_{-}(t)\leq -k_{\hat{c}}(t),\ k_{\hat{c}}(t)\leq g_{+}(t)\ \text{for}\ t\in(0,+\infty),\\
 w_{\hat{c}}(t,x)\leq u(t,x)\ \text{for}\ t\in(0,+\infty), x\in(-k_{\hat{c}}(t),k_{\hat{c}}(t)).
\end{aligned}
\right.
\end{equation}
If
\begin{equation}\label{initialcompri}
g(0)\geq k_{\hat{c}}(0)=g_{+}(0), v_{0}(x)>w_{\hat{c}}(0,x)=u(0,x)\ \text{on}\ [-k_{\hat{c}}(0),k_{\hat{c}}(0)],
\end{equation}
then for $t\geq 0$, we have
\begin{equation}\label{compare2}
\left\{
\begin{aligned}
g_{+}(t)\leq g(t)\ &\text{for}\ t\in(0,+\infty),\\
u(t,x)\leq v(t,x)\ &\text{for}\ t\in(0,+\infty), x\in(g_{-}(t),g_{+}(t)).
\end{aligned}
\right.
\end{equation}
by using Lemma \ref{lem4.1} with $\overline{u}(t,x)=v(t,x), \overline{g}_{-}(t)=g_{-}(t), \overline{g}_{+}(t)=g(t)$.
It follows from \eqref{compare1} and \eqref{compare2} that 
$$-k_{\hat{c}}(t)<-\hat{c}(t)<-ct,\ ct<\hat{c}(t)<k_{\hat{c}}(t)\leq g(t)\ \text{for}\ t\in(0,+\infty),$$
$$ w_{\hat{c}}(t,x)\leq v(t,x)\ \text{for}\ t\in(0,+\infty), x\in[-k_{\hat{c}}(t),k_{\hat{c}}(t)]\supset[-ct,ct].$$
Hence \eqref{4.3} holds.
Now by the almost same arguments as used in the proof of \cite[Lemma 6.5]{DLo}, we can obtain \eqref{4.4}.

If \eqref{initialcompri} fails, then we can consider $(\tilde{v}(t,x),\tilde{g}(t))=(v(t,x-G),g(t)+G)$
for some $G>0$.
Hence, for $G$ large enough, \eqref{initialcompri} holds for $(\tilde{v},\tilde{g})$ since $\liminf\limits_{x\to-\infty}v_{0}(x)\geq1$.
Thus \eqref{4.3} and \eqref{4.4} hold for $(\tilde{v},\tilde{g})$,
thereby hold for $(v,g)$ since $c$ is arbitrary.

The proof of \eqref{upperbdd} is essentially the same as the proof of (iii) of \cite[Lemma 6.5]{DLo}.
\end{proof}

\begin{rem}\label{re4.3}
Assume that $f$ is of type $(f_{M})$,
and that $(v(t,x),g(t))$ is a solution of \eqref{4.2} with $T=+\infty$. If
$0\leq\liminf\limits_{x\to-\infty}v_{0}(x)\leq\limsup\limits_{x\to-\infty}v_{0}(x)<+\infty$,
then the conclusions of Lemma \ref{lem4.3} still hold. In fact, taking a solution $(u,g_{-},g_{+})$
of \eqref{4.1} with $g_{0}\geq\pi/\big(2\sqrt{f^{\prime}(0)}\big)$ and
$\sup\limits_{x\in(-g_{0},g_{0})}u_{0}(x)<\liminf\limits_{x\to-\infty}v_{0}(x)$,
we have $g(t)+G\geq g_{+}(t)$ and $v(t,x-G)\geq u(t,x)$ for $t>0, x\in[g_{-}(t),g_{+}(t)]$ with some $G\geq0$ by Lemma \ref{lem4.1}. Then, it follows from
\cite[Corollary 4.5 and Lemma 6.5]{DLo} that \eqref{4.3}--\eqref{upperbdd} hold.
\end{rem}

\begin{lemma}\label{lem4.4}
Assume that \eqref{f} holds, that $(c^{*},q_{c^{*}})$ is the solution of \eqref{2.1}, and that
$(u(t,x),h(t))$ is a bounded transition semi-wave of \eqref{1.1} which connects $1$ and $0$. Then $u(t,x)\in[0,1]$ for $t\in\mathbb{R}, x\leq h(t)$. 
Moreover, for any $c\in(0,c^{*})$, there exist $\delta\in(0,-f^{\prime}(1)), T^{*}>0$ and $M>0$ such that for $t\geq T^{*}$,
\begin{equation}\label{l4.2.1}
{u}(t,x)\geq1-Me^{-\delta t}\ \text{for}\ x\leq ct.
\end{equation}
\end{lemma}

\begin{proof}
First we show that $u(t,x)\in[0,1]$. It suffices to show that $u(t,x)\leq1$ for $t\in\mathbb{R}, x\leq h(t)$. If not, then there exists $(\xi,\tau)$ with $\xi<h(\tau)$
such that $u(\xi,\tau)>1$. 
Now for any $t_{0}\in\Real$, let $(\tilde{u}(t,x),\tilde{h}(t))=\big(u(t+t_{0},x+h(t_{0})),h(t+t_{0})-h(t_{0})\big)$.
Then $(\tilde{u}(t,x),\tilde{h}(t))$ is still a bounded transition semi-wave of \eqref{1.1} connecting $1$ and $0$ and
the assumptions in Lemma \ref{lem4.3} hold for $(\tilde{u}(t,x),\tilde{h}(t))$.
From \eqref{upperbdd}, we have $\tilde{u}(t,x)\leq1+Me^{-\delta t}\ \text{for}\ t\geq T^{*}, x\leq \tilde{h}(t).$
Moreover, the $T^{*}$ here does not depend on $t_{0}$ since $\inf\limits_{t\in\mathbb{R}}h^{\prime}(t)>0$ and $\lim\limits_{x\to-\infty}|u(t,x+h(t))-1|=0$ uniformly in $t\in\Real$.
By enlarging $T^{*}$ we may assume that $Me^{-\delta t}<u(\tau,\xi)-1$ for $t\geq T^{*}$.
Letting $t_{0}=\tau-T^{*}$, we have
$$u(\tau,\xi)=u(T^{*}+t_{0},\xi-h(t_{0})+h(t_{0}))=\tilde{u}(T^{*},\xi-h(t_{0}))\leq1+Me^{-\delta t}<u(\tau,\xi),$$
which is a contradiction. Hence $u(t,x)\leq1$ for $t\in\mathbb{R}, x\leq h(t)$.

Next we show that \eqref{l4.2.1} holds. Choose $\eta>0$ so small that $f^{\prime}(u)<0$ for $u\in[1,1+\eta]$.
Consider $\tilde{f}(u)$ decreasing in $[1,+\infty)$ with $\tilde{f}(u)=f(u)$ in $[0,1+\eta]$. Then $(u,h)$ is a transition semi-wave of \eqref{1.1} with $f(t,x,u)$
replaced by $\tilde{f}(u)$. Hence Proposition \ref{propspaceshift} implies \eqref{l4.2.1} because of \eqref{4.4}.
\end{proof}

\begin{lemma}\label{lem4.5}
Let the assumptions of Lemma \ref{lem4.4} hold.
Then $|h(t)-c^{*}t|$ is bounded for $t\in\mathbb{R}$.
\end{lemma}
\begin{proof}
For any $t_{0}\in\Real$, let $(\tilde{u}(t,x),\tilde{h}(t))$ be as in the proof of Lemma \ref{lem4.4}.
Note that $\delta\in(0,-f^{\prime}(1))$. Then there exists $\eta>0$ such that
\begin{equation*}
\left\{
   \begin{aligned}
  \delta\leq-f^{\prime}(u),\ \ &u\in[1-\eta,1+\eta],\\
   f(u)\geq0,\ \ &u\in[1-\eta,1].\\
\end{aligned}
   \right.
\end{equation*}
By enlarging $T^{*}$ we may assume that $Me^{-\delta t}<\frac{\eta}{2}$ for $t\geq T^{*}$.
We take $M^{\prime}>M$ such that $M^{\prime}e^{-\delta T^{*}}<\eta$.
We can also find $X_{0}>0$ such that
$$(1+M^{\prime}e^{-\delta T^{*}})q_{c^*}(X_{0})\geq1+Me^{-\delta T^{*}}$$
since $q_{c^{*}}(x)\to1$ as $x\to1$.
Now take
\begin{equation*}
\left\{
   \begin{aligned}
  \overline{h}(t)=c^{*}(t-T^{*})+\sigma M^{\prime}(e^{-\delta T^{*}}-e^{-\delta t})+\tilde{h}(T^{*})+X_{0},\\
   \overline{u}(t,x)=(1+M^{\prime}e^{-\delta t})q_{c^{*}}(\overline{h}(t)-x),\\
\end{aligned}
   \right.
\end{equation*}
and
\begin{equation*}
\left\{
   \begin{aligned}
  \underline{h}(t)=c^{*}(t-T^{*})+cT^{*}-\sigma M(e^{-\delta T^{*}}-e^{-\delta t}),\\
   \underline{u}(t,x)=(1-Me^{-\delta t})q_{c^{*}}(\underline{h}(t)-x).\\
\end{aligned}
   \right.
\end{equation*}
Computing as Lemmas 3.2 and 3.3 in \cite{DMZ}, we can show that, for $\sigma$ large enough and $t\geq T^{*}$,
$(\overline{u}(t,x),\overline{h}(t))$ and $(\underline{u}(t,x),\underline{h}(t))$
are upper and lower solutions of $(\tilde{u}(t,x),\tilde{h}(t))$, respectively.
We mention here that we need Lemma \ref{lem4.4} to show that $(\underline{u}(t,x),\underline{h}(t))$ is a lower solution.
Hence by Lemma \ref{lem4.2}, we have
$\underline{h}(t)\leq\tilde{h}(t)\leq\overline{h}(t)$ for any $t\geq T^{*},$ which yields that for any
$t\geq T^{*},$
$$(c-c^{*})T^{*}-\sigma M-B-g_{0}\leq h(t+t_{0})-h(t_{0})-c^{*}t\leq-c^{*}T^{*}+\sigma M^{\prime}+\tilde{h}(T^{*})+X_{0}-B-g_{0}.$$
Note that $t_{0}$ is arbitrary. Then the last inequality becomes
\begin{equation}\label{4.13}
C_{1}\leq h(t)-h(\tau)-c^{*}(t-\tau)\leq C_{2}.
\end{equation}
for any $\tau\in\Real, t\geq\tau+T^{*},$ where $C_{1}=(c-c^{*})T^{*}-\sigma M-B-g_{0}$ and
$C_{2}=\sigma M^{\prime}-c^{*}T^{*} +\tilde{h}(T^{*})+X_{0}-B-g_{0}.$
Setting $t=0, \tau\leq-T^{*}$, we have
$$-C_{2}\leq h(\tau)-c^{*}\tau-h(0)\leq-C_{1},$$
i.e., $-C_{2}+h(0)\leq h(t)-c^{*}t\leq-C_{1}+h(0)$ for any $t\leq-T^{*}.$
Setting $\tau=0, t\geq T^{*}$ in \eqref{4.13}, we have
$C_{1}+h(0)\leq h(t)-c^{*}t\leq C_{2}+h(0).$ Therefore, $|h(t)-c^{*}t|$ is bounded.
\end{proof}

\begin{lemma}\label{lem4.6}
Let the assumptions of Lemma \ref{lem4.3} hold. Then $|g(t)-c^{*}t|$ is bounded for $t\geq0$.
Moreover, if $f$ is of type $f_{M}$, then the condition 
$1\leq\liminf\limits_{x\to-\infty}v_{0}(x)\leq\limsup\limits_{x\to-\infty}v_{0}(x)<+\infty$
can be replaced by 
$0\leq\liminf\limits_{x\to-\infty}v_{0}(x)\leq\limsup\limits_{x\to-\infty}v_{0}(x)<+\infty$.
\end{lemma}
\begin{proof}
We only need to prove that the conclusion holds if $v_{0}(x)$ is decreasing since we can find two
decreasing smooth functions $\underline{v}_{0}$ and $\overline{v}_{0}$ such that $\underline{v}_{0}\leq v_{0}\leq\overline{v}_{0}.$ 
Then we complete the proof by using Lemma \ref{lem4.2}.

Since $v_{0}$ is decreasing, we obtain that $v(t,x)$ is decreasing in $x$ for any fixed $t$ by Remark \ref{re4.2}.
Therefore, for any $c\in(0,c^{*})$, there exist
$\delta\in(0,-f^{\prime}(1)), T^{*}>0$ and $M>0$ such that
$v(t,x)\geq1-Me^{-\delta t}\ \text{for}\ x\leq ct$ and $t\geq T^{*}$.
Then an argument similar to the proof of
Lemma \ref{lem4.5} gives us the conclusion.
\end{proof}

\subsection{End the proof of Theorem \ref{thmhomo}}

\begin{proof}[Proof of Theorem \ref{thmhomo}]
Assume that $|h(t)-c^{*}t|\leq A$ since $|h(t)-c^{*}|$ is bounded.
Let $(v(t,x),b(t))=(u(t,x+c^{*}t),h(t)-c^{*}t)$. Then $(v(t,x),b(t))$ satisfies
\begin{equation}\label{4.14}
\left\{
   \begin{aligned}
  v_{t}=v_{xx}+c^{*}v_{x}+f(v),\ \ &t\in\Real, x<b(t),\\
   v(t,b(t))=0,\ \ &t\in\Real,\\
   b^{\prime}(t)=-\mu v_{x}(t,b(t))-c^{*},\ \ &t\in\Real.\\
   \end{aligned}
   \right.
\end{equation}
By enlarging $B$ we may assume that $u(t,x)\geq1-\frac{\theta}{2}$ for $x-h(t)\leq-B$, i.e.,
$v(t,x)\geq1-\frac{\theta}{2}$ for $x-b(t)\leq-B$.
For any $\tau\in\Real,\ \xi\geq0$, denote $v^{\tau,\xi}(t,x):=v(t+\tau,x-\xi)$.\\
Step 1: Show that for any $\tau\in\Real$ and $\xi\geq B+2A$, $v^{\tau,\xi}(t,x)\geq v(t,x)$ for $x\leq b(t).$\\
First, note that $v^{\tau,\xi}(t,x)=v(t+\tau,x-\xi)\geq 1-\frac{\theta}{2}$ for $x\leq b(t)$
since $(x-\xi)-b(t+\tau)\leq b(t)-B-2A-b(t+\tau)\leq-B.$ Set
$$\varepsilon^{*}=\inf\{\varepsilon>0: v^{\tau,\xi}(t,x)+\varepsilon\geq v(t,x)\ \forall t\in\Real, x\leq b(t)\}.$$
Then by the same argument as used in the proof of Proposition \ref{propspaceshift}, we have $\varepsilon^{*}=0$, i.e.,
$v^{\tau,\xi}(t,x)\geq v(t,x)$ for $x\leq b(t).$

Now, for any fixed $\tau\in\Real$, let us define
$$\xi^{*}=\inf\{\xi>0: v^{\tau,\xi^{\prime}}(t,x)\geq v(t,x)\ \forall t\in\Real, x\leq b(t),
\xi^{\prime}\geq\xi\}.$$
Then $\xi^{*}\in[0,B+2A]$, $v^{\tau,\xi^{*}}(t,x)\geq v(t,x)$ for any $t\in\Real, x\leq b(t)$,
and $b(t+\tau)+\xi^{*}-b(t)\geq0$. Moreover,
$v^{\tau,\xi^{*}}(t,b(t+\tau)+\xi^{*})=v(t+\tau,b(t+\tau))=0.$
We want to show that $\xi^{*}=0$.\\
Step 2: Show that $\inf\limits_{t\in\mathbb{R}}\{b(t+\tau)+\xi^{*}-b(t)\}>0$.\\
Suppose that $\inf\limits_{t\in\mathbb{R}}\{b(t+\tau)+\xi^{*}-b(t)\}=0$.
Then there are two cases we need to consider:\\
Case 1: There exists $t_{0}$ such that $b(t_{0}+\tau)+\xi^{*}-b(t_{0})=0$.\\
Note that $b(t+\tau)+\xi^{*}-b(t)\geq0$ for any $t\in\Real$.
Then $b^{\prime}(t_{0}+\tau)-b^{\prime}(t_{0})=0$. Hence
\begin{equation}\label{4.15}
v^{\tau,\xi^{*}}(t_{0},b(t_{0}))=v_{x}(t_{0}+\tau,b(t_{0})-\xi^{*})=v_{x}(t_{0}+\tau,b(t_{0}+\tau))=v_{x}(t_{0},b(t_{0})),
\end{equation}
where the last equality follows from \eqref{4.14}. On the other hand, we have
$v^{\tau,\xi^{*}}(t,x)>v(t,x)$ for $x<b(t)$.
In fact, if there exists $(s,y)$ with $y<b(s)$ such that $v^{\tau,\xi^{*}}(s,y)=v(s,y)$,
then the strong maximum principle yields that
$v^{\tau,\xi^{*}}(t,x)\equiv v(t,x)$ for $x\leq b(t)$. In particular, from this, we have
$$v(t+n\tau,b(t)-n\xi^{*})=v(t,b(t))=0,\ \forall n\in\mathbb{N}.$$
But $\lim\limits_{n\to\infty}v(t+n\tau,b(t)-n\xi^{*})=\lim\limits_{n\to\infty}u(t+n\tau,h(t)-n\xi^{*})=1$. 
Therefore, $w(t,x)>0$ for $x<b(t)$.
By Hopf's Lemma, we have $w_{x}(t_{0},b(t_{0}))<0$, i.e.,
$v^{\tau,\xi^{*}}(t_{0},b(t_{0}))<v_{x}(t_{0},b(t_{0})),$ which contradicts \eqref{4.15}.
Thus Case 1 can not occur.\\
Case 2: There exists $\{t_{n}\}_{n\in\mathbb{N}}$ such that $\lim\limits_{n\to\infty}\big(b(t_{n}+\tau)+\xi^{*}-b(t_{n})\big)=0$.\\
Let $$b_{n}(t)=b(t+t_{n})-b(t_{n}),$$
$$v_{n}(t,x)=v(t+t_{n},x+b(t_{n}))\ \text{for}\ x\leq b_{n}(t).$$
Then using an argument similar to the one in the proof of Step 3 in Theorem \ref{thm3.1},
we know that Case 2 can not occur either.
Hence $\inf\limits_{t\in\mathbb{R}}\{b(t+\tau)+\xi^{*}-b(t)\}>0$.\\
Step 3: Show that $\xi^{*}=0$.\\
Claim: We have $\inf\limits_{-B\leq x-b(t)\leq0}\{v^{\tau,\xi^{*}}(t,x)-v(t,x)\}>0.$\\
Proof of Claim: If not, then there exists a sequence $\{(t_{n},x_{n})\}_{n\in\mathbb{N}}$ with $-B\leq x_{n}-b(t_{n})\leq0$
such that $v^{\tau,\xi^{*}}(t_{n},x_{n})-v(t_{n},x_{n})\to0$ as $n\to\infty$.
Note that $\inf\limits_{t\in\mathbb{R}}v^{\tau,\xi^{*}}(t,b(t))>0=v(t,b(t))$ since $\inf\limits_{t\in\mathbb{R}}\{b(t+\tau)+\xi^{*}-b(t)\}>0$ and $v_{xx}$ is uniformly bounded.
Then there exists $\sigma,\kappa>0$ such that
$v^{\tau,\xi^{*}}(t,x)-v(t,x)>\sigma$ for $-2\kappa\leq x-b(t)\leq0.$
Hence we may assume $-B\leq x_{n}-b(t_{n})<-2\kappa$. Let $w=v^{\tau,\xi^{*}}-v$.
Then $w$ satisfies
\begin{equation}\label{eqofw}
w_{t}=w_{xx}+\frac{f(v^{\tau,\xi^{*}})-f(v)}{v^{\tau,\xi^{*}}-v}w\ \text{for}\ x-b(t)<0,
\end{equation}
$w(t,x)\geq0$ for $x-b(t)\leq0$, and $w(t_{n},x_{n})\to 0$ as $n\to\infty$.
Take $K\in\mathbb{Z}^{+}$ with $\kappa K>B+2b_{0},$
where $b_{0}=\sup\limits_{t\in\mathbb{R}}|b^{\prime}(t)|$. For $i=0,1,\cdots,K-1$, we set
$t_{n}^{i}, x_{n}^{i},$ and $E_{n}^{i}$ as the proof of Proposition \ref{prop3.2}. By the similar arguments there,
we can still obtain a contradiction. Therefore the proof of claim is complete.

Suppose that $\xi^{*}>0$. Now by the claim above, there exists $\xi_{0}\in(0,\xi^{*})$ such that
$v^{\tau,\xi^{*}-\xi}(t,x)\geq v(t,x)$ for any $\xi\in[0,\xi_{0}], -B\leq x-b(t)\leq0.$
Note that
$v(t,x)\geq1-\frac{\theta}{2}$ for any $ x-b(t)\leq-B.$
Then for $\xi_{0}$ sufficiently small,
$$v^{\tau,\xi^{*}-\xi}(t,x)\geq1-\theta$$
for any $\xi\in[0,\xi_{0}], x-b(t)\leq-B$ since $v_{x}$ is uniformly bounded. Now setting
$$\varepsilon^{*}=\inf\{\varepsilon>0: v^{\tau,\xi^{*}-\xi}(t,x)+\varepsilon\geq v(t,x)\ \forall t\in\Real, x-b(t)\leq-B\},$$
we have $\varepsilon^{*}=0$ by the same argument as used in the proof of Proposition \ref{propspaceshift}. Then
$v^{\tau,\xi^{*}-\xi}(t,x)\geq v(t,x)$ for any $\xi\in[0,\xi_{0}], x-b(t)\leq-B.$
Hence $v^{\tau,\xi^{*}-\xi}(t,x)\geq v(t,x)$ for any $\xi\in[0,\xi_{0}], x-b(t)\leq0,$
which contradicts the definition of $\xi^{*}$.
Therefore, $\xi^{*}=0.$\\
Step 4: Up to now, we have proved that
$v^{\tau,0}(t,x)\geq v(t,x)$, i.e., $v(t+\tau,x)\geq v(t,x)$ for $(t,\tau)\in\Real^{2}$
and $x\leq b(t)$. Then $v$ is independent of $t$.
Taking the derivative of the second equation of \eqref{4.14} with respect to $t$, we have
$$v_{x}(t,b(t))b^{\prime}(t)=0.$$
Then $b^{\prime}(t)=0$ since $v_{x}(t,b(t))<0$. Therefore, $b(t)\equiv\text{constant}.$
By Theorem \ref{thm4.1}, $v=q_{c^{*}}$ up to a translation.
\end{proof}

\section{Proof of Theorems \ref{heterotime} and \ref{heterospace}}

This section is devoted to proving Theorems \ref{heterotime} and \ref{heterospace}.

\begin{proof}[Proof of Theorem \ref{heterotime}]
	Let $w(t,x)=u(t,x+h(t))$. Then $w$ satisfies \eqref{fixedbd}.
	It is sufficient to show $w(t,x)=\phi(t,x+\zeta(t))$,
	where $(\phi(t,x),\zeta(t))$ is the time almost periodic positive semi-wave solution of \eqref{time} with $\zeta(0)=0$.
	An observation is that $u(t,x)<u_{c}(t)$ for $t\in\mathbb{R}$.
	In fact,
	we can regard $u(t;M)$, which is a solution of \eqref{timeode}
	with initial value $M$ large enough, as a upper solution of $u(t,x)$ (extended by 0 on $\{x>h(t)\}$) since $u(t,x)$ is bounded.
	Then by using \cite[Remark 3.2]{LLS}, we can easily obtain that $u(t,x)<u_{c}(t)$ for $t\in\mathbb{R}$.
	
	Denote $\Phi(t,x)=\phi(t,x+\zeta(t)).$
	Note that $w(t,x)=u(t,x+h(t))<u_{c}(t)$. Then by \cite[Lemma 5.3]{LLS}, we have
	$$w(t,x)\leq\Phi(t,x)\ \forall t\in\Real, x<0.$$
	Suppose that $w(t,x)\lvertneqq\Phi(t,x)$. Then it follows from \cite[Lemma 4.6]{LLS} that
	$$\rho(w(t_{2},\cdot),\Phi(t_{2},\cdot))<\rho(w(t_{1},\cdot),\Phi(t_{1},\cdot))\
	\forall t_{1}<t_{2},$$
	since $\lim\limits_{x\to-\infty}w(t,x)=u_{c}(t)=\lim\limits_{x\to-\infty}\Phi(t,x).$
	Let $\rho_{-}=\lim\limits_{t\to-\infty}\rho(w(t,\cdot),\Phi(t,\cdot)).$
	Then $\rho_{-}>0.$ Let $\{t_{n}\}_{n\in\mathbb{N}}$ be a sequence with $t_{n}\to-\infty$.
	Denote $c_{n}(t,u)=c(t+t_{n},u), w_{n}(t,x)=w(t+t_{n},x)$, and $\Phi_{n}(t,x)=\Phi(t+t_{n},x)$.
	Then there exists a subsequence, still denoted by $\{t_{n}\}_{n\in\mathbb{N}}$, such that
	$c_{n}(t)\to c^{*}(t), w_{n}(t,x)\to w^{*}(t,x),\ \text{and}\
	\Phi_{n}(t,x)\to \Phi^{*}(t,x)$
	with $c^{*}\in H(c), w^{*}(t,\cdot), \Phi^{*}(t,\cdot)\in X$ for any $t\in\Real.$
	Obviously, we have $w^{*}(t,x)\leq \Phi^{*}(t,x)$.
	Then $\rho(w^{*}(t,\cdot),\Phi^{*}(t,\cdot))$ is well defined and
	\begin{equation}\label{6.2}
	\rho(w^{*}(t,\cdot),\Phi^{*}(t,\cdot))\equiv\rho_{-}>0.
	\end{equation}
	On the other hand, $w^{*}$ and $\Phi^{*}$ satisfy
	\begin{equation*}
	\left\{
	\begin{aligned}
	w^{*}_{t}=w^{*}_{xx}-\mu w^{*}_{x}(t,0)w^{*}_{x}+w^{*}(c^{*}(t)-w^{*}),\ \ &t\in\Real, x<0,\\
	w^{*}(t,0)=0,\ \ &t\in\Real,\\
	\end{aligned}
	\right.
	\end{equation*}
	and
	\begin{equation*}
	\left\{
	\begin{aligned}
	\Phi^{*}_{t}=\Phi^{*}_{xx}-\mu \Phi^{*}_{x}(t,0)\Phi^{*}_{x}+\Phi^{*}(c^{*}(t)-\Phi^{*}),\ \ &t\in\Real, x<0,\\
	\Phi^{*}(t,0)=0,\ \ &t\in\Real,\\
	\end{aligned}
	\right.
	\end{equation*}
	respectively.
	By \cite[Lemma 4.6]{LLS}, $\rho(w^{*}(t,\cdot),\Phi^{*}(t,\cdot))$ is strictly decreasing in $t$
	since $\lim\limits_{x\to-\infty}w^{*}(t,x)=u_{c}(t)=\lim\limits_{x\to-\infty}\Phi^{*}(t,x)$,
	which  contradicts \eqref{6.2}.
	Hence we must have $\rho(w(t,\cdot),\Phi(t,\cdot))\equiv0$, i.e., $w(t,x)\equiv\Phi(t,x)$.
\end{proof}

\begin{proof}[Proof of Theorem \ref{heterospace}] This conclusion follows (iii) of \cite[Theorem 1.1]{L} directly. 

	In fact, if $(u(t,x),h(t))$ is a transition semi-wave of \eqref{space} which connects $v_{a}(x)$ and $0$,
	then $\lim\limits_{x\to-\infty}|u(t,x+h(t))-v_{a}(x+h(t))|=0$ uniformly in $t\in\Real$.
	Hence for $\varepsilon_{0}=\frac{1}{2}\inf\limits_{x\in\Real}v_{a}(x)>0$,
	there exists $\delta>0$ such that $u(t,x+h(t))\geq v_{a}(x+h(t))-\varepsilon_{0}\geq\varepsilon_{0}$
	for any $t\in\Real, x\leq\delta,$ i.e., $u(t,x)\geq\varepsilon_{0}$ for any $t\in\Real, x\in(-\infty,h(t)-\delta]$.
	Note that $u$ is bounded. Using \cite[Proposition 1.1]{L}, we can easily show that $u(t,x)\leq v_{a}(x)$.
	Then it follows from (iii) of \cite[Theorem 1.1]{L} that $(u(t,x),h(t))$ is the almost periodic semi-wave solution of \eqref{space}.
\end{proof}

\section{A transition semi-wave without any global mean speeds}

In this section, we prove the existence of transition semi-waves, which do not have global mean speeds, for some special heterogeneous equations.
First, we prove a theorem we will need later.

\begin{thm}\label{thm5.1}
Let the assumptions of Lemma \ref{lem4.3} hold. Then there exists $\hat{G}\in\mathbb{R}$ such that
$$\lim\limits_{t\to+\infty}(g(t)-c^{*}t-\hat{G})=0,
\lim\limits_{t\to+\infty}g^{\prime}(t)=c^{*},$$
and
$$\lim\limits_{t\to+\infty}\sup\limits_{x\leq g(t)}|v(t,x)-q_{c^{*}}(g(t)-x)|=0.$$
Moreover, if $f$ is of type $(f_{M})$, then the condition 
$1\leq\liminf\limits_{x\to-\infty}v_{0}(x)\leq\limsup\limits_{x\to-\infty}v_{0}(x)<+\infty$ 
can be replaced by 
$0\leq\liminf\limits_{x\to-\infty}v_{0}(x)\leq\limsup\limits_{x\to-\infty}v_{0}(x)<+\infty$.
\end{thm}
\begin{proof}
The strategy of proof is almost the same as that used in subsections 3.2 and 3.3 in \cite{DMZ}. We only point out the outline of the proof.
Setting $\tilde{v}(t,x)=v(t,x+c^{*}t)$, $F(s)=\int_{0}^{s}f(\theta)d\theta$, and
$$E(t)=\int_{-\infty}^{g(t)-c^{*}t}e^{c^{*}z}\{\frac{1}{2}\tilde{v}^{2}_{x}-F(\tilde{v})\}dx,$$
we can prove as subsection 3.2 in \cite{DMZ} to obtain that for any sequence $\{t_{n}\}_{n\in\mathbb{N}}$ with $t_{n}\to+\infty$,
there exists a subsequence $\{\tilde{t}_{n}\}_{n\in\mathbb{N}}\subset\{t_{n}\}_{n\in\mathbb{N}}$ such that
$\lim\limits_{n\to\infty}\big(g(\tilde{t}_{n}+\cdot)-c^{*}(\tilde{t}_{n}+\cdot)\big)=\hat{G}$ in
$C_{loc}^{1}(\mathbb{R})$ for some constant $\hat{G}\in\mathbb{R}$. Moreover,
$\lim\limits_{n\to\infty}\sup\limits_{x\leq\hat{G}}|\tilde{v}(\tilde{t}_{n},x)-q_{c^{*}}(\hat{G}-x)|=0$.
Here we have followed the convention that $q_{c^{*}}(x)=0$ for $x\leq0$ and $\tilde{v}(t,x)=0$ for $x\geq g(t)-c^{*}t.$
Now using the upper and lower solutions constructed in subsection 3.3 in \cite{DMZ}, we obtain the conclusions we need.
\end{proof}

Let us turn our attention to constructing a transition semi-wave for some special $f$.
More precisely, assume that $f$ satisfies:
\begin{equation}\label{5.1}
\left\{
   \begin{aligned}
  &(1) f(t,x,0)=f(t,x,1)=0\ \text{for}\ (t,x)\in\mathbb{R}^{2},\\
  &(2) \exists t_{1}<t_{2}\in\mathbb{R},\exists f_{1},f_{2}\ \text{satisfying}\ \eqref{f}\ \text{s.t.}\\
  &\ \ \ f(t,s)=f_{1}(s)\ \forall (t,s)\in(-\infty,t_{1}]\times[0,1]\ \text{and}\\
  &\ \ \ f(t,s)=f_{2}(s)\ \forall (t,s)\in(t_{2},+\infty]\times[0,1].
   \end{aligned}
   \right.
\end{equation}

\begin{example}\label{ex6.1}
Assume that \eqref{5.1} holds, and that there exists $(c_{i}^{*},q_{c_{i}^{*}})$ satisfying \eqref{2.1} with $f$ replaced by $f_{i}$, $i=1,2$. Then there exists a transition semi-wave of \eqref{1.1}, which connects $1$ and $0$.
Moreover, $h(t)=c^{*}_{1}t$ for $t\leq t_{1}$ and $\lim\limits_{t\to+\infty}h(t)/t=c_{2}^{*}.$
\end{example}
\begin{proof}
Taking $(u(t,x),h(t))=(q_{c^{*}_{1}}(c^{*}_{1}t-x),c^{*}_{1}t)$  for $t\leq t_{1}$ and $x\leq c^{*}_{1}t$,
and then letting it evolve as time goes on, we can obtain an entire solution $(u,h)$ of \eqref{1.1}.
Next we only need to show that $(u,h)$ is the solution what we need exactly.

Obviously, we have
\begin{equation}\label{5.2}
\lim\limits_{x\to-\infty}|u(t,x+c^{*}_{1}t)-1|=0\ \text{uniformly w.r.t.}\ t\leq t_{1}.
\end{equation}
Let us now study the behavior of $u$ on the time interval $[t_{1}, +\infty)$.
From the strong parabolic maximum principle, there holds $0<u(t,x)<1$ for all
$t\in\mathbb{R}, x\leq h(t)$.
Furthermore, from standard parabolic estimates, the function $u$ satisfies the limiting conditions
$$u(t,-\infty)=1\ \text{locally in}\ t\in\mathbb{R},$$
since $f(t,x,1)\equiv0.$ In particular,
\begin{equation}\label{5.3}
\lim\limits_{x\to-\infty}|u(t,x+h(t))-1|=0\ \text{uniformly w.r.t.}\ t\in (t_{1},t_{2}].
\end{equation}
Now regard $(u,h)$ as a solution of
\begin{equation*}
\left\{
   \begin{aligned}
  u_{t}=u_{xx}+f(t,x,u),\ \ &t>t_{2},\ x<h(t),\\
   h^{\prime}(t)=-\mu u_{x}(t,h(t)),\ \  &t>t_{2},\\
   u(t,h(t))=0,\ \ \ &t>t_{2},\\
   \end{aligned}
   \right.
\end{equation*}
which starts at time $t_{2}$. Then by Theorem \ref{thm5.1}, we have
\begin{equation}\label{5.4}
\lim\limits_{t\to+\infty}\sup\limits_{x\leq h(t)}|u(t,x)-q_{c_{2}^{*}}(h(t)-x)|=0.
\end{equation}
It follows from \eqref{5.3} and \eqref{5.4} that
$$\lim\limits_{x\to-\infty}|u(t,x+h(t))-1|=0\ \text{uniformly w.r.t.}\ t\in [t_{2},+\infty).$$
Combining the above equation with \eqref{5.2} and \eqref{5.3}, we konw that $(u,h)$ is a transition semi-wave connecting 1 and 0.
The proof is thereby complete.
\end{proof}

\begin{rem}\label{re5.1}

\begin{enumerate}
\item[1)] There does not exist a global mean speed if $c^{*}_{1}\neq c^{*}_{2}$.
	Moreover, $f(t,x,s)$ can be viewed as a local perturbation of a homogeneous equation if $f_{1}=f_{2}$.
\item[2)] Let $\mathcal{F}=\{f\ \text{satisfies}\ \eqref{f}:\eqref{2.1}\ \text{possesses a solution}\ (c^{*},q_{c^{*}})\}$.
	Then whatever $f_{1}\in\mathcal{F}$ may be and whatever the profile of $f$, which satisfies \eqref{5.1}, between times $t_{1}$ and $t_{2}$ may be,
	the speed of the position $h(t)$ at large time is determined only by $f_{2}$.
\item[3)] If $f_{1}$ in \eqref{5.1} equals to $a(x)u(1-u)$ for some positive periodic function $a(x)$,
	then we can still construct a transition semi-wave connecting $1$ and $0$ such that
	$h^{\prime}(t)$ equals to a periodic function for $t\leq t_{1}$ and $\lim\limits_{t\to+\infty}h(t)/t=c_{2}^{*}.$
	In fact, using \cite[Theorem 1.2]{DLiang}, we can prove it by an argument very similar to the proof of Example \ref{ex6.1}.
\end{enumerate}	
\end{rem}

\end{document}